\documentclass[12pt]{article}

\usepackage{latexsym,amsmath,amscd,amssymb,graphics,float}
\usepackage{enumerate}

\usepackage{graphicx,lscape}

\usepackage[colorlinks]{hyperref}
\usepackage{url}

\begin{document}

\newenvironment{proof}[1][Proof]{\textbf{#1.} }{\ \rule{0.5em}{0.5em}}

\newtheorem{theorem}{Theorem}[section]
\newtheorem{definition}[theorem]{Definition}
\newtheorem{lemma}[theorem]{Lemma}
\newtheorem{remark}[theorem]{Remark}
\newtheorem{proposition}[theorem]{Proposition}
\newtheorem{corollary}[theorem]{Corollary}
\newtheorem{example}[theorem]{Example}

\numberwithin{equation}{section}
\newcommand{\ep}{\varepsilon}
\newcommand{\R}{{\mathbb  R}}
\newcommand\C{{\mathbb  C}}
\newcommand\Q{{\mathbb Q}}
\newcommand\Z{{\mathbb Z}}
\newcommand{\N}{{\mathbb N}}

\newcommand{\bfi}{\bfseries\itshape}

\newsavebox{\savepar}
\newenvironment{boxit}{\begin{lrbox}{\savepar}
\begin{minipage}[b]{15.5cm}}{\end{minipage}\end{lrbox}
\fbox{\usebox{\savepar}}}

\title{{\bf On differential operators generated by geometric structures}}
\author{R\u{a}zvan M. Tudoran}

\date{}
\maketitle \makeatother

\begin{abstract}
We extend to manifolds endowed with a general geometric structure, the classical notions of gradient as well as Laplace operator, and provide some of their natural properties.
\end{abstract}

\medskip



\section{Introduction}
\label{section:one}

The main purpose of this article is to introduce and study some natural differential operators generated by geometric structures on general manifolds, where by geometric structure we mean any non--degenerate $(0,2)-$tensor field (e.g., Riemannian or semi--Riemannian metrics, Lorentzian metrics, symplectic structures). More precisely, the first class of operators, extend to manifolds endowed with geometric structures, the concepts of gradient vector field from Riemannian or semi--Riemannian geometry, and Hamiltonian vector field from symplectic geometry, while the second class of operators, generalize to manifolds with geometric structures the Laplace--Beltrami operator from Riemannian geometry, as well as the d'Alembert operator from Lorentzian geometry. On vector spaces, these classes of operators were recently introduced in \cite{TDR}.

The structure of the article is the following. In the second section we introduce the concept of geometric structure on a general manifold, the notion of geometromorphism (i.e., a diffeomorphism that preserves geometric structures), and provide some of their natural properties. Also here we introduce the notion of left/right--adjoint of a $(1,1)-$tensor field with respect to a geometric structure. The third section is devoted to the study of two of the main protagonists of this work, namely, the left/right--gradient vector fields, naturally associated to a general geometric structure. These vector fields extend to geometric manifolds the notion of gradient vector field from the Riemannian and semi--Riemannian geometry, as well as the Hamiltonian vector field from the symplectic geometry. Using these special vector fields, we prove the existence of a natural Leibniz bracket on each manifold endowed with a geometric structure, and we also rediscover the left/right--gradient vector fields as being the left/right--Leibniz vector fields generated by this bracket. Moreover, various relations between the gradient--like vector fields on a manifold with geometric structure and the Leibniz vector fields associated to the symmetric or skew--symmetric part of the above mentioned Leibniz bracket, are also given. In the fourth section we define left/right--Laplace operators generated by a general geometric structure on a manifold, and study some of their natural properties. Among other things, we provide Green's identities for these operators, and give the Euler--Lagrange equation for Dirichlet energy associated to a general geometric structure. The last section of this article is devoted to a dynamical approach of the gradient--like vector fields, associated to a general geometric structure. Here we discuss the compatibility of these vector fields with the Leibniz bracket of a manifold endowed with geometric structure, we give a natural transport theorem, and provide a non--existence result regarding periodic orbits.

We end the introduction by mentioning that throughout this article, all manifolds are assumed to be connected, Hausdorff, paracompact, and smooth. Moreover, in order to have a unitary presentation, and also to avoid intricate formulations, all mappings are assumed to be smooth.

\section{Geometric structures on a manifold and some natural properties}

In this section we introduce the concept of \textit{geometric structure} on a general manifold, and provide some of their natural properties which will be of crucial importance on the rest of the article.

\subsection{Geometric structures}

Let us start by introducing the notion of \textit{geometric structure}, the concept on which relies the whole construction of the article. This is a generalization of the similar notion from the linear framework, recently introduced in \cite{TDR}.

\begin{definition} 
We call \textbf{geometric structure} on a manifold $M$, any non-degenerate $(0,2)-$tensor field $b:\mathfrak{X}(M)\times\mathfrak{X}(M)\rightarrow\mathcal{C}^{\infty}(M,\mathbb{R})$. The pair $(M,b)$ will be called a \textbf{geometric manifold}, or a \textbf{manifold with geometric structure}.
\end{definition}
If $x=(x^1,\dots,x^n)$ denote a set of local coordinates on $M$, where $n=\operatorname{dim}M$, then a geometric structure $b$ is locally represented by the non--singular matrix 
\begin{equation*}
\mathcal{B}_x:=(\mathcal{B}_{ij}(x))_{1\leq i,j\leq n}, ~~ \mathcal{B}_{ij}:=b\left(\dfrac{\partial}{\partial x^i},\dfrac{\partial}{\partial x^j}\right).
\end{equation*}
Notice that Riemannian, semi--Riemannian, symplectic, almost symplectic, Lorentzian structures, are all particular examples of geometric structures. Of course, in order to exist, most geometric structures require additional properties of the ambient manifold (e.g., an almost symplectic structure on a manifold requires the manifold be orientable and even-dimensional; a Lorentzian structure requires the manifold be either non compact or compact with zero Euler characteristic, etc.).

Let us point out a simple method to generate more geometric structures, starting from a given one. More precisely, on a general geometric manifold, $(M,b)$, \textbf{any} non-degenerate $(1,1)-$tensor field, $T:\mathfrak{X}(M)\rightarrow\mathfrak{X}(M)$, \textbf{induces two more geometric structures}, $b_{T}^{L},b_{T}^{R}:\mathfrak{X}(M)\times\mathfrak{X}(M)\rightarrow\mathcal{C}^{\infty}(M,\mathbb{R})$, given by
\begin{align*}
b_{T}^{L}(X,Y):=b(TX,Y),~\forall X,Y\in\mathfrak{X}(M); ~ b_{T}^{R}(X,Y):=b(X,TY),~\forall X,Y\in\mathfrak{X}(M).
\end{align*}
Recall that particular examples of non--degenerate $(1,1)-$tensor fields include almost compex structures, almost product structures, almost golden structures, etc.

Conversely, to any given pair of geometric structures defined on a manifold, one can naturally associate various non--degenerate $(1,1)-$tensor fields, essentially, following backwards the previous construction. An important particular case is provided by a geometric structure defined on a Riemannian manifold. 
\begin{remark}\label{RI1}
If $(M,g)$ is a Riemannian manifold and $b$ a geometric structure on $M$, then the non-degeneracy of $b$ implies the existence of a non-degenerate $(1,1)-$tensor field $B_g :\mathfrak{X}(M)\rightarrow\mathfrak{X}(M)$ uniquely defined by the relation
$$
g(X,Y)=b(X,B_{g}Y), ~\forall X,Y\in\mathfrak{X}(M).
$$
The pair $(b,B_g)$ will be called the \textbf{geometric pair} associated to the geometric structure $b$.
\end{remark}
As already mentioned, the above construction can be generalized mimetically to the case of an arbitrary fixed geometric manifold, together with an additional geometric structure $b$. However, we focused on this particular case, because of the rich geometry offered by the Riemannian settings, thus making it easier to analyze the geometry of $b$ with respect to it.

Let us now give a local representation of the quantities introduced in Remark \ref{RI1}. Thus, relative to a set of local coordinates, $x=(x^1,\dots,x^n)$, where $n=\operatorname{dim}M$, the metric $g$ is represented by the symmetric and positive definite matrix 
\begin{equation*}
G_x:=(g_{ij}(x))_{1\leq i,j\leq n}, ~ g_{ij}=g\left(\dfrac{\partial}{\partial x^i},\dfrac{\partial}{\partial x^j}\right),
\end{equation*}
the geometric structure $b$ is represented by the non--singular matrix 
\begin{equation*}
\mathcal{B}_x:=(\mathcal{B}_{ij}(x))_{1\leq i,j\leq n}, ~ \mathcal{B}_{ij}=b\left(\dfrac{\partial}{\partial x^i},\dfrac{\partial}{\partial x^j}\right),
\end{equation*}
and the tensor field $B_g$ is represented by the non--singular matrix 
\begin{equation*}
B_x:=(B^{i}_{j}(x))_{1\leq i,j\leq n}, ~ B^{i}_{j}(x)=\sum_{k=1}^{n}\mathcal{B}^{ik}(x)g_{kj}(x), ~ where ~ (\mathcal{B}^{ij}(x))_{1\leq i,j\leq n}:=\mathcal{B}^{-1}_x.
\end{equation*}

\subsection{Geometromorphisms}

Next step that follows naturally after introducing geometric manifolds is to see how this notion is preserved through maps between such manifolds. In order to do this, we define a class of diffeomorphisms, which preserve geometric structures. These maps will suggestively be called \textit{geometromorphisms}, they becoming isometries in the category of Riemannian, semi--Riemannian, or Lorentzian geometric structures, and symplectomorphism in the category of symplectic geometric structures.

\begin{definition}
Let $(M,b^M)$ and $(N,b^N)$ be two manifolds endowed with geometric structures. A \textbf{geometromorphism} from $(M,b^M)$ to $(N,b^N)$ is a diffeomorphism $\Phi:M\rightarrow N$ which preserves the geometric structures, i.e., $\Phi^{\star}b^N=b^M$,
where
\begin{align*}
(\Phi^{\star}b^{N})(X,Y):=\Phi^{\star}(b^{N}(\Phi_{\star}X,\Phi_{\star}Y)), ~\forall X,Y\in\mathfrak{X}(M),
\end{align*}
$\Phi_{\star}X:=\mathrm{D}\Phi\circ X\circ\Phi^{-1}\in\mathfrak{X}(N), ~\forall X\in\mathfrak{X}(M)$, and $\Phi^{\star}F:=F\circ\Phi\in\mathcal{C}^{\infty}(M,\mathbb{R}), ~\forall F\in\mathcal{C}^{\infty}(N,\mathbb{R})$.
\end{definition}

In local coordinates, $x=(x^1,\dots,x^n)$ (where $n=\operatorname{dim}M$), if $(\mathcal{B}^{M}_{ij}(x))_{1\leq i,j\leq n}$ and $(\mathcal{B}^{N}_{kl}(\Phi(x)))_{1\leq k,l\leq n}$ denote the local representations of the geometric structures $b^M$ and $b^N$ respectively, the local representation of the geometromorphism $\Phi$ has to satisfy
\begin{align*}
\mathcal{B}^{M}_{ij}(x)=\sum_{1\leq k,l\leq n}\mathcal{B}^{N}_{kl}(\Phi(x))\dfrac{\partial \Phi^k}{\partial x^i}\dfrac{\partial \Phi^l}{\partial x^j}, ~\forall i,j\in\{1,\dots,n\}.
\end{align*}

Next, we prove that the set of all geometromorphisms of a given manifold $M$ endowed with a geometric structure $b$, forms a group, called the \textit{group of geometromorphisms} of $(M,b)$.
\begin{theorem}
Let $b$ be a geometric structure on a manifold $M$. Then the set of all geometromorphisms of $(M,b)$, denoted by $\mathcal{G}(M,b):=\{\Phi\in\operatorname{Diff}(M)\mid \Phi^{\star}b=b\}$, is a subgroup of the group of diffeomorphisms of $M$.
\end{theorem}
\begin{proof}
By the very definition of geometromorphisms we have the inclusion $\mathcal{G}(M,b)\subseteq \operatorname{Diff}(M)$. Notice that the identity map, $\operatorname{Id}_{M}$, is obviously a geometromorphism. Next, we prove that for any $\Phi,\Psi\in\mathcal{G}(M,b)$, their composition, $\Phi\circ\Psi$, belongs to $\mathcal{G}(M,b)$. Indeed, as $\Phi^{\star}b=b$ and $\Psi^{\star}b=b$ it follows that
$$
(\Phi\circ\Psi)^{\star}b=\Psi^{\star}(\Phi^{\star}b)=\Psi^{\star}b=b.
$$
In order to finish the proof we show that for any $\Phi\in\mathcal{G}(M,b)$ the inverse map, $\Phi^{-1}\in\operatorname{Diff}(M)$, is also an element of $\mathcal{G}(M,b)$. This follows directly taking into account that $\Phi^{\star}b=b$ and hence
$$
(\Phi^{-1})^{\star}b=(\Phi^{-1})^{\star}(\Phi^{\star}b)=(\Phi\circ\Phi^{-1})^{\star}b=(\operatorname{Id}_{M})^{\star}b=b.
$$
\end{proof}

\subsection{The left/right--adjoint of a $(1,1)-$tensor field}

The aim of this subsection is to extend to geometric manifolds the notion of adjoint associated to $(1,1)-$tensor fields from the Riemannian or semi--Riemannian setting.

\begin{definition}\label{starLR}
Let $b$ be a geometric structure on a manifold $M$, and $A:\mathfrak{X}(M)\rightarrow\mathfrak{X}(M)$ an $(1,1)-$tensor field. Then there exist two $(1,1)-$tensor fields, denoted by $A^{\star_{L}},A^{\star_{R}}:\mathfrak{X}(M)\rightarrow\mathfrak{X}(M)$ and called the \textit{left-adjoint}, respectively the \textit{right-adjoint} of $A$ with respect to the geometric structure $b$, uniquely defined by the following relations:
\begin{equation*}
b(A^{\star_{L}}X, Y)= b(X, A Y), ~~\forall  X,Y\in\mathfrak{X}(M),
\end{equation*}
\begin{equation*}
b(A X, Y)= b(X, A^{\star_{R}}Y), ~~\forall  X,Y\in\mathfrak{X}(M).
\end{equation*}
\end{definition}
In local coordinates, $x=(x^1,\dots,x^n)$, where $n=\operatorname{dim}M$, the above tensor fields are represented by the matrices
\begin{equation*}
A^{\star_{L}}_{x}=\mathcal{B}_{x}^{-\top}A_{x}^{\top}\mathcal{B}_{x}^{\top}, ~~ A^{\star_{R}}_{x}=\mathcal{B}_{x}^{-1}A_{x}^{\top}\mathcal{B}_{x}.
\end{equation*}
In the case when $b$ is a geometric structure on a Riemannian manifold $(M,g)$, since $G_x=\mathcal{B}_{x}B_x$, the above local expressions can be equivalently written as
\begin{equation*}
A^{\star_{L}}_{x}=G_{x}^{-1}B_{x}^{\top}A_{x}^{\top}B_{x}^{-\top}G_{x}, ~~ A^{\star_{R}}_{x}=B_{x}G_{x}^{-1}A_{x}^{\top}G_{x}B_{x}^{-1}.
\end{equation*}
Note that the left and right adjoint (of a $(1,1)-$tensor field $A$) with respect to the Riemannian metric tensor $g$ coincide, and are locally given by
\begin{equation*}
A^{\star_{g}}_{x}=G_{x}^{-1}A_{x}^{\top}G_{x}, ~ \text{as in this case} ~ B_x=I_n.
\end{equation*}

Let us provide now the main properties of the left/right--adjoint operators associated to a geometric manifold.
\begin{proposition}
Let $(M,b)$ be a geometric manifold, $A,B:\mathfrak{X}(M)\rightarrow\mathfrak{X}(M)$ two $(1,1)-$tensor fields, and $F,G\in\mathcal{C}^{\infty}(M,\mathbb{R})$. Then the following equalities hold
\begin{itemize}
\item[(i)] $(FA+GB)^{\star_{L}/\star_{R}}=FA^{\star_{L}/\star_{R}}+GB^{\star_{L}/\star_{R}}$,
\item[(ii)] $(A^{\star_{L}})^{\star_{R}}=A, ~~ (A^{\star_{R}})^{\star_{L}}=A$,
\item[(iii)] $(AB)^{\star_{L}/\star_{R}}=B^{\star_{L}/\star_{R}}A^{\star_{L}/\star_{R}}$,
\item[(iv)] $(A^{-1})^{\star_{L}/\star_{R}}=(A^{\star_{L}/\star_{R}})^{-1}$, if $A$ is invertible,
\item[(v)] $A^{\star_{L}}=A^{\star_{R}}$, if $b$ is symmetric or skew--symmetric.
\end{itemize}
\end{proposition}
\begin{proof}
All relations follow directly from Definition \ref{starLR}.
\end{proof}

In the case of a geometric structure defined on a Riemannian manifold, some specific compatibilities between the associated adjoint operators are presented below.
\begin{remark}\label{UY}
Let $b$ be a geometric structure on a Riemannian manifold $(M,g)$, and $(b,B_g)$ the associated geometric pair. Then $B_{g}^{\star_{L}}= B_{g}^{\star_{g}}$, and $B_{g}^{\star_{R}}= B_{g}B_{g}^{\star_{g}}B_{g}^{-1}$, where $B_{g}^{\star_g}$ stands for the adjoint of $B_g$ with respect to the Riemannian metric $g$. Moreover, if $b$ is symmetric (skew--symmetric) then $B_{g}^{\star_{g}}=B_{g}$ ($B_{g}^{\star_{g}}= - B_{g}$).
\end{remark}
In local coordinates, $x=(x^1,\dots,x^n)$, where $n=\operatorname{dim}M$, the tensor fields from Remark \ref{UY} are represented by the matrices
\begin{equation*}
B^{\star_{L}}_{x}=G_{x}^{-1}B_{x}^{\top}G_{x}=B_{x}^{\star_g}, ~~ B^{\star_{R}}_{x}=B_{x}G_{x}^{-1}B_{x}^{\top}G_{x}B_{x}^{-1}=B_{x}B_{x}^{\star_g}B_{x}^{-1}.
\end{equation*}
Note that, as $\mathcal{B}_x =G_x B^{-1}_{x}$, the following relations hold
\begin{itemize}
\item[(i)] \text{if} $b$ \text{is symmetric, then} $\mathcal{B}_{x}^{\top}=\mathcal{B}_x \Leftrightarrow G^{-1}_{x}B^{\top}_{x}=B_{x}G^{-1}_{x}$,
\item[(ii)] \text{if} $b$ \text{is skew--symmetric, then} $\mathcal{B}_{x}^{\top}=-\mathcal{B}_x \Leftrightarrow G^{-1}_{x}B^{\top}_{x}=-B_{x}G^{-1}_{x}$.
\end{itemize}

\section{Gradient--like vector fields generated by geometric structures}

In this section we introduce two of the main protagonists of this work, namely, the left/right--gradient vector fields, naturally associated to a general geometric structure.  These vector fields extend to geometric manifolds the notion of gradient vector field from the Riemannian and semi--Riemannian geometry, as well as the Hamiltonian vector field from the symplectic geometry. On vector spaces endowed with constant geometric structures, the gradient--like vector fields were recently  introduced in \cite{TDR}.

\begin{definition}
Let $(M,b)$ be a geometric manifold, and $U \subseteq M$ an open set. Then for every $F\in\mathcal{C}^{\infty}(U,\mathbb{R})$, the left--gradient of $F$, denoted by $\nabla^{L}_{b}F$, is the vector field uniquely defined by the relation
$$
b(\nabla^{L}_{b}F,X)=\mathrm{d}F\cdot X, ~ \forall X\in\mathfrak{X}(U).
$$
Similarly, the right--gradient of $F$, denoted by $\nabla^{R}_{b}F$, is the vector field uniquely defined by the relation
$$
b(X,\nabla^{R}_{b}F)=\mathrm{d}F\cdot X, ~ \forall X\in\mathfrak{X}(U).
$$
\end{definition}
Note that $\mathcal{C}^1$ is the sufficient regularity class in order the definition of left/right--gradients hold. Nevertheless, we choose to work in the smooth class, for having a unitary approach throughout the article, and also to avoid intricate formulations.

In terms of local coordinates, $x=(x^1,\dots,x^n)$, where $n=\operatorname{dim}M$, the left/right--gradient vector fields are represented by the matrices
\begin{equation*}
(\nabla^{L}_{b}F)_{x}=\mathcal{B}^{-\top}_{x}(\mathrm{d}F)_{x}, ~~ (\nabla^{R}_{b}F)_{x}=\mathcal{B}^{-1}_{x}(\mathrm{d}F)_{x},
\end{equation*}
where $(\mathrm{d}F)_x:=\left(\dfrac{\partial F}{\partial x^1}(x)~\dots~\dfrac{\partial F}{\partial x^n}(x)\right)^{\top}$. Using the notation $(\mathcal{B}^{ij}(x))_{1\leq i,j\leq n}:=\mathcal{B}^{-1}_x$, we obtain
\begin{equation}\label{grdloc}
\nabla^{L}_{b}F=\sum_{i=1}^{n}\left(\sum_{j=1}^{n}\mathcal{B}^{ji}\dfrac{\partial F}{\partial{x^{j}}}\right)\dfrac{\partial}{\partial{x^{i}}}, ~~ \nabla^{R}_{b}F=\sum_{i=1}^{n}\left(\sum_{j=1}^{n}\mathcal{B}^{ij}\dfrac{\partial F}{\partial{x^{j}}}\right)\dfrac{\partial}{\partial{x^{i}}}.
\end{equation}

\begin{remark}\label{REMI}
Note that if $b$ is symmetric (i.e., $(M,b)$ is a Riemannian or semi--Riemannian manifold) then $\nabla^{L}_{b}=\nabla^{R}_{b}$. Particularly, when $b=g$ is a Riemannian metric then $\nabla^{L}_{b}=\nabla^{R}_{b}=\nabla_{g}$, where $\nabla_g$ stands for the gradient operator on the Riemannian manifold $(M,g)$, defined for any function $F\in\mathcal{C}^{\infty}(U,\mathbb{R})$ by the relation
\begin{equation*}
g(\nabla_{g} F,X)=\mathrm{d}F\cdot X, ~ \forall X\in\mathfrak{X}(U).
\end{equation*}
On the other hand, if $b$ is skew--symmetric (i.e., $(M,b)$ is an almost symplectic manifold) then $\nabla^{L}_{b}= - \nabla^{R}_{b}$. If moreover, $\mathrm{d}b=0$ (i.e., $(M,b)$ is a symplectic manifold), then for any $H\in\mathcal{C}^{\infty}(M,\mathbb{R})$, the associated Hamiltonian vector field, $X_H$ (defined by the relation $\mathbf{i}_{X_H}b=\mathrm{d}H$), coincides with $\nabla_{b}^{L}H=-\nabla_{b}^{R}H$.
\end{remark}

In the case when $b$ is a general geometric structure on a Riemannian manifold, $(M,g)$, the relation between the operators $\nabla^{L}_{b}$, $\nabla^{R}_{b}$ is given by the following result which is a direct consequence of the definition of $\nabla^{L}_{b}$, $\nabla^{R}_{b}$, and $\nabla_g$.

\begin{proposition}\label{rlgrd}
Let $b$ be a geometric structure on a Riemannian manifold $(M,g)$, and let $(b,B_g)$ be the associated geometric pair. Then for any fixed open set $U\subseteq M$, and for every $F\in\mathcal{C}^{\infty}(U,\mathbb{R})$, we have that $\nabla^{L}_{b}F=B_{g}^{\star_{g}}\nabla_{g} F$, $\nabla^{R}_{b}F=B_{g} \nabla_{g} F$, and $\nabla^{L}_{b}F=B_{g}^{\star_{g}}B_{g}^{-1}\nabla^{R}_{b}F$.
\end{proposition}
In local coordinates, $x=(x^1,\dots,x^n)$, where $n=\operatorname{dim}M$, as $G_x=\mathcal{B}_{x}B_x$, the left/right--gradient vector fields are represented by the matrices
\begin{equation*}
(\nabla^{L}_{b}F)_{x}=\mathcal{B}^{-\top}_{x}(\mathrm{d}F)_{x}=G^{-1}_{x}B^{\top}_{x}(\mathrm{d}F)_x, ~~ (\nabla^{R}_{b}F)_{x}=\mathcal{B}^{-1}_{x}(\mathrm{d}F)_{x}=B_{x}G^{-1}_{x}(\mathrm{d}F)_x.
\end{equation*}

\subsection{Gradient--like vector fields and geometromorphisms}

This subsection is devoted to analyzing the compatibility between left/right--gradient vector fields and geometromorphisms.

\begin{theorem}\label{grdinvar}
Let $(M,b^M)$, $(N,b^{N})$ be two manifolds endowed with geometric structures, and let $\Phi:M\rightarrow N$ be a geometromorphism. Then the following relation holds true
\begin{equation}\label{gradrel}
\Phi_{\star}\circ\nabla^{L/R}_{b^{M}}\circ\Phi^{\star}=\nabla^{L/R}_{b^{N}},
\end{equation}
or explicitly,
\begin{align*}
\Phi_{\star}(\nabla^{L/R}_{b^{M}}(\Phi^{\star}F))=\nabla^{L/R}_{b^{N}}F, ~\forall F\in\mathcal{C}^{\infty}(N,\mathbb{R}).
\end{align*}
\end{theorem}
\begin{proof} Let $F\in\mathcal{C}^{\infty}(N,\mathbb{R})$ be arbitrary fixed. We will prove first the identity concerning the left-gradient operator. In order to do this, let $Y\in\mathfrak{X}(N)$ be arbitrary fixed. Then, as $\Phi:M\rightarrow N$ is a geometromorphism (i.e., $\Phi^{\star}b^N =b^M$), we have
\begin{align*}
b^{N}(\Phi_{\star}(\nabla^{L}_{b^{M}}(\Phi^{\star}F)),Y)&=b^{N}(\Phi_{\star}(\nabla^{L}_{b^{M}}(\Phi^{\star}F)),\Phi_{\star}(\Phi^{-1}_{\star}Y))\\
&=(\Phi^{\star})^{-1}(\Phi^{\star}(b^{N}(\Phi_{\star}(\nabla^{L}_{b^{M}}(\Phi^{\star}F)),\Phi_{\star}(\Phi^{-1}_{\star}Y))))\\
&=(\Phi^{\star})^{-1}((\Phi^{\star}b^{N})(\nabla^{L}_{b^{M}}(\Phi^{\star}F),\Phi^{-1}_{\star}Y))\\
&=(\Phi^{\star})^{-1}(b^{M}(\nabla^{L}_{b^{M}}(\Phi^{\star}F),\Phi^{-1}_{\star}Y))\\
&=(\Phi^{\star})^{-1}(\mathrm{d}(\Phi^{\star}F)\cdot\Phi^{-1}_{\star}Y)=(\Phi^{\star})^{-1}(\Phi^{\star}(\mathrm{d}F\cdot Y))\\
&=\mathrm{d}F\cdot Y=b^{N}(\nabla^{L}_{b^N}F,Y).
\end{align*}
Thus, we obtained
\begin{align*}
b^{N}(\Phi_{\star}(\nabla^{L}_{b^{M}}(\Phi^{\star}F)),Y)=b^{N}(\nabla^{L}_{b^N}F,Y), ~\forall Y\in\mathfrak{X}(N),
\end{align*} 
and by the non-degeneracy of $b^{N}$, we get $\Phi_{\star}(\nabla^{L}_{b^{M}}(\Phi^{\star}F))=\nabla^{L}_{b^{N}}F$. Since this relation holds for any $F\in\mathcal{C}^{\infty}(N,\mathbb{R})$, we conclude that
$$
\Phi_{\star}\circ\nabla^{L}_{b^{M}}\circ\Phi^{\star}=\nabla^{L}_{b^{N}}.
$$

The identity for the right--gradient operator follows similarly. More exactly, for any $Y\in\mathfrak{X}(N)$, we have
\begin{align*}
b^{N}(Y,\Phi_{\star}(\nabla^{R}_{b^{M}}(\Phi^{\star}F)))&=b^{N}(\Phi_{\star}(\Phi^{-1}_{\star}Y),\Phi_{\star}(\nabla^{R}_{b^{M}}(\Phi^{\star}F)))\\
&=(\Phi^{\star})^{-1}(\Phi^{\star}(b^{N}(\Phi_{\star}(\Phi^{-1}_{\star}Y),\Phi_{\star}(\nabla^{R}_{b^{M}}(\Phi^{\star}F)))))\\
&=(\Phi^{\star})^{-1}((\Phi^{\star}b^{N})(\Phi^{-1}_{\star}Y,\nabla^{R}_{b^{M}}(\Phi^{\star}F)))\\
&=(\Phi^{\star})^{-1}(b^{M}(\Phi^{-1}_{\star}Y,\nabla^{R}_{b^{M}}(\Phi^{\star}F)))\\
&=(\Phi^{\star})^{-1}(\mathrm{d}(\Phi^{\star}F)\cdot\Phi^{-1}_{\star}Y)=(\Phi^{\star})^{-1}(\Phi^{\star}(\mathrm{d}F\cdot Y))\\
&=\mathrm{d}F\cdot Y=b^{N}(Y,\nabla^{R}_{b^N}F),
\end{align*}
and by the non-degeneracy of $b^{N}$, we obtain $\Phi_{\star}(\nabla^{R}_{b^{M}}(\Phi^{\star}F))=\nabla^{R}_{b^{N}}F$. As this relation holds for any $F\in\mathcal{C}^{\infty}(N,\mathbb{R})$, we get
$$
\Phi_{\star}\circ\nabla^{R}_{b^{M}}\circ\Phi^{\star}=\nabla^{R}_{b^{N}}.
$$
\end{proof}

A particular case of Theorem \ref{grdinvar} which worth mentioned, occurs when $(M,b^M)=(N,b^{N})$. More precisely, the following result holds.

\begin{corollary}\label{grdgeome}
Let $(M,b)$ be a manifold endowed with a geometric structure, and $\mathcal{G}(M,b)$ the associated group of geometromorphisms. Then
\begin{equation}\label{gradge}
\Phi_{\star}\circ\nabla^{L/R}_{b}=\nabla^{L/R}_{b}\circ(\Phi^{-1})^{\star}, ~\forall \Phi\in\mathcal{G}(M,b),
\end{equation}
or explicitly,
\begin{align*}
\Phi_{\star}(\nabla^{L/R}_{b}F)=\nabla^{L/R}_{b}((\Phi^{-1})^{\star}F), ~\forall F\in\mathcal{C}^{\infty}(M,\mathbb{R}),~ \forall \Phi\in\mathcal{G}(M,b).
\end{align*}
\end{corollary}
Let us point out now that for any manifold with geometric structure, $(M,b)$, there are two natural actions of the group of geometromorphisms:
\begin{itemize}
\item[(i)] $\tau(\Phi)\bullet F:=(\Phi^{-1})^{\star}F,~ \forall F\in\mathcal{C}^{\infty}(M,\mathbb{R}), ~\forall \Phi\in\mathcal{G}(M,b),$
\item[(ii)] $\tilde{\tau}(\Phi)\bullet X:=\Phi_{\star}X,~ \forall X\in\mathfrak{X}(M), ~\forall \Phi\in\mathcal{G}(M,b),$
\end{itemize}
given by the restriction to $\mathcal{G}(M,b)$ of the corresponding actions of the group $\operatorname{Diff}(M)$. 
In terms of the above defined actions, relation \eqref{gradge} says actually that the operators $\nabla^{L}_{b}, \nabla^{R}_{b}$ are $\mathcal{G}(M,\mathbb{R})-$equivariant, namely 
\begin{equation}\label{grdequiv}
\tilde{\tau}(\Phi)\bullet \nabla^{L/R}_{b}F=\nabla^{L/R}_{b}(\tau(\Phi)\bullet F), ~\forall F\in\mathcal{C}^{\infty}(M,\mathbb{R}), ~\forall \Phi\in\mathcal{G}(M,b).
\end{equation}

\subsection{Brackets associated to a geometric structure}

In this subsection we introduce and study several brackets naturally associated to a geometric structure. Those brackets reveals another feature of geometric manifolds, proving that geometric manifolds are particular cases of Leibniz manifolds. Moreover, each geometric manifold can be given naturally three  Leibniz structures. The existence of these brackets is tight related to gradient--like vector fields generated by geometric structures. For details regarding Leibniz manifolds, see e.g., \cite{JP}.

Let us provide now a simple result which constitutes actually the starting point of the above mentioned constructions.
\begin{proposition}\label{fimpo}
Let $b$ be a geometric structure on a manifold $M$, and $U \subseteq M$ an open set. Then the following identity holds
\begin{equation}\label{supimp}
b(\nabla_{b}^{L}F,\nabla_{b}^{L}G)=b(\nabla_{b}^{R}F,\nabla_{b}^{R}G), ~\forall F,G \in\mathcal{C}^{\infty}(U,\mathbb{R}).
\end{equation}
\end{proposition}
\begin{proof}
The relation \eqref{supimp} follows directly from the definition of left/right--gradient vector fields induced by the geometric structure $b$. Indeed, for any $F,G \in\mathcal{C}^{\infty}(U,\mathbb{R})$ we obtain
\begin{align*}
b(\nabla_{b}^{L}F,\nabla_{b}^{L}G)=\mathrm{d}F\cdot\nabla_{b}^{L}G=b(\nabla_{b}^{L}G,\nabla_{b}^{R}F)=\mathrm{d}G\cdot\nabla_{b}^{R}F=b(\nabla_{b}^{R}F,\nabla_{b}^{R}G).
\end{align*}
\end{proof}

The Proposition \ref{fimpo} implies the existence of a natural Leibniz bracket on each geometric manifold, $(M,b)$, which \textit{will be called the $b-$bracket}, thus proving that each geometric manifold is an example of Leibniz manifold. This will be formalized in the following result.

\begin{theorem}\label{leib2}
Let $b$ be a geometric structure on a manifold $M$. The mapping $\{\cdot,\cdot\}_b:\mathcal{C}^{\infty}(M,\mathbb{R})\times\mathcal{C}^{\infty}(M,\mathbb{R})\rightarrow \mathcal{C}^{\infty}(M,\mathbb{R})$, given by
$$
\{F,G\}_b:=b(\nabla_{b}^{L}F,\nabla_{b}^{L}G)=b(\nabla_{b}^{R}F,\nabla_{b}^{R}G), ~\forall F,G\in\mathcal{C}^{\infty}(M,\mathbb{R}),
$$
is a Leibniz bracket (i.e., an $\mathbb{R}-$bilinear mapping on the algebra $\mathcal{C}^{\infty}(M,\mathbb{R})$, which is a derivation on each entry), and consequently $(M,\{\cdot,\cdot\}_b)$ is a Leibniz manifold.
\end{theorem}
\begin{proof}
As the $\mathbb{R}-$bilinearity of $\{\cdot,\cdot\}_{b}$ is a direct consequence of the $\mathbb{R}-$linearity of the left/right--gradient operators, we will prove only the derivation property of $\{\cdot,\cdot\}_{b}$ on the first entry, the derivation property in the second entry following mimetically. Indeed, for any $F,G,H\in\mathcal{C}^{\infty}(M,\mathbb{R})$, we have
\begin{align*}
\{FG,H\}_{b}&=b(\nabla_{b}^{L}(FG),\nabla_{b}^{L}H)=b(F\nabla_{b}^{L} G+G \nabla_{b}^{L} F,\nabla_{b}^{L} H)\\
&=Fb(\nabla_{b}^{L} G,\nabla_{b}^{L} H)+Gb(\nabla_{b}^{L} F,\nabla_{b}^{L} H)\\
&=F\{G,H\}_{b}+G\{F,H\}_{b}.
\end{align*}
\end{proof}

\begin{example}\label{symplbr}
Recall from Remark \ref{REMI} that if $b$ is skew--symmetric and $\mathrm{d}b=0$ (i.e., $(M,b)$ is a symplectic manifold), then for any $H\in\mathcal{C}^{\infty}(M,\mathbb{R})$, the associated Hamiltonian vector field, $X_H$, coincides with $\nabla_{b}^{L}H=-\nabla_{b}^{R}H$. Thus, in this case, the $b-$bracket coincides with the Poisson bracket induced by the symplectic form $b$, i.e., $\{F,G\}_{b}=b(X_F ,X_G), ~\forall F,G\in\mathcal{C}^{\infty}(M,\mathbb{R})$. Recall that a Poisson bracket on the algebra $\mathcal{C}^{\infty}(M,\mathbb{R})$ is a a skew-symmetric Leibniz bracket which verifies the Jacobi identity. For details regarding Poisson brackets, see e.g., \cite{MR}, \cite{RTC}.
\end{example}

Let us now give a local representation of the $b-$bracket. More precisely, in terms of the local coordinates, $x=(x^1,\dots,x^n)$, where $n=\operatorname{dim}M$, the $b-$bracket becomes
\begin{equation*}
\{F,G\}_b (x)=(\mathrm{d}F)_{x}^{\top}\mathcal{B}_{x}^{-\top}(\mathrm{d}G)_{x},
\end{equation*}
or equivalently
\begin{equation*}
\{F,G\}_{b}=\sum_{1\leq i,j\leq n}\mathcal{B}^{ji}\dfrac{\partial F}{\partial{x^{i}}}\dfrac{\partial G}{\partial{x^{j}}},
\end{equation*}
where $(\mathcal{B}^{ij}(x))_{1\leq i,j\leq n}:=\mathcal{B}^{-1}_x$.

In the following result we show that the gradient--like vector fields generated by a geometric structure $b$, actually coincide with the Leibniz vector fields generated by the $b-$bracket. 
\begin{proposition}
Let $b$ be a geometric structure on a manifold $M$, $\{\cdot,\cdot\}_{b}$ the associated $b-$bracket, and $F\in\mathcal{C}^{\infty}(M,\mathbb{R})$. Then the left--Leibniz vector field generated by $F$ (i.e.,  $G\in\mathcal{C}^{\infty}(M,\mathbb{R})\mapsto \{G,F\}_{b}\in\mathcal{C}^{\infty}(M,\mathbb{R})$) coincides with $\nabla_{b}^{L}F$, while the right--Leibniz vector field generated by $F$ (i.e.,  $G\in\mathcal{C}^{\infty}(M,\mathbb{R})\mapsto \{F,G\}_{b}\in\mathcal{C}^{\infty}(M,\mathbb{R})$) coincides with $\nabla_{b}^{R}F$.
\end{proposition}
\begin{proof} Let us fix $F\in\mathcal{C}^{\infty}(M,\mathbb{R})$. Then the left--Leibniz vector field,
\begin{align*}
G\in\mathcal{C}^{\infty}(M,\mathbb{R})\mapsto \{G,F\}_{b}=b(\nabla_{b}^{L} G,\nabla_{b}^{L} F)\in\mathcal{C}^{\infty}(M,\mathbb{R}),
\end{align*}
can be equivalently written as
\begin{align*}
G\in\mathcal{C}^{\infty}(M,\mathbb{R})\mapsto \mathrm{d}G\cdot\nabla_{b}^{L} F=(\nabla_{b}^{L} F) (G)\in\mathcal{C}^{\infty}(M,\mathbb{R}),
\end{align*}
hence, it coincides with $\nabla_{b}^{L} F$.

Similarly, the right--Leibniz vector field,
\begin{align*}
G\in\mathcal{C}^{\infty}(M,\mathbb{R})\mapsto \{F,G\}_{b}=b(\nabla_{b}^{R} F,\nabla_{b}^{R} G)\in\mathcal{C}^{\infty}(M,\mathbb{R}),
\end{align*}
can be equivalently written as
\begin{align*}
G\in\mathcal{C}^{\infty}(M,\mathbb{R})\mapsto \mathrm{d}G\cdot\nabla_{b}^{R} F=(\nabla_{b}^{R} F) (G)\in\mathcal{C}^{\infty}(M,\mathbb{R}),
\end{align*}
and consequently, it is exactly $\nabla_{b}^{R} F$.
\end{proof}

Let us now present another compatibility between geometric structures and the associated brackets, this time in terms of geometromorphisms. More exactly, we shall prove that geometromorphisms preserve $b-$brackets.
\begin{theorem}\label{yuo}
Let $(M,b^M)$, $(N,b^{N})$ be two manifolds endowed with geometric structures, and let $\Phi:M\rightarrow N$ be a geometromorphism. Then
\begin{equation*}
\Phi^{\star}\{F,G\}_{b^{N}}=\{\Phi^{\star}F,\Phi^{\star}G\}_{b^M}, ~\forall F,G\in\mathcal{C}^{\infty}(N,\mathbb{R}),
\end{equation*}
where $\{\cdot,\cdot\}_{b^M},\{\cdot,\cdot\}_{b^N}$ denote the $b-$brackets associated to the geometric structures $b^M$ and $b^N$, respectively.
\end{theorem}
\begin{proof} As $\Phi:M\rightarrow N$ is a geometromorphism, and $(\Phi^{-1})^{\star}=(\Phi^{\star})^{-1}$, it follows that $b^N=(\Phi^{-1})^{\star}b^M$. Using this observation and the Theorem \ref{grdinvar}, we obtain for any $F,G\in\mathcal{C}^{\infty}(N,\mathbb{R})$
\begin{align*}
\{\Phi^{\star}F,\Phi^{\star}G\}_{b^M}&=b^{M}(\nabla^{L/R}_{b^M}(\Phi^{\star}F),\nabla^{L/R}_{b^M}(\Phi^{\star}G))=b^{M}(\Phi^{-1}_{\star}(\nabla^{L/R}_{b^{N}}F),\Phi^{-1}_{\star}(\nabla^{L/R}_{b^{N}}G))\\
&=\Phi^{\star}((\Phi^{\star})^{-1}(b^{M}(\Phi^{-1}_{\star}(\nabla^{L/R}_{b^{N}}F),\Phi^{-1}_{\star}(\nabla^{L/R}_{b^{N}}G))))\\
&=\Phi^{\star}(((\Phi^{-1})^{\star}b^{M})(\nabla^{L/R}_{b^{N}}F,\nabla^{L/R}_{b^{N}}G))=\Phi^{\star}(b^{N}(\nabla^{L/R}_{b^{N}}F,\nabla^{L/R}_{b^{N}}G))\\
&=\Phi^{\star}\{F,G\}_{b^N},
\end{align*}
and thus we get the conclusion.
\end{proof}

A particular case of Theorem \ref{yuo} which worth mentioned, occurs when $(M,b^M)=(N,b^{N})$. More precisely, the following result holds.
\begin{corollary}\label{corimp}
Let $(M,b)$ be a manifold endowed with a geometric structure, and $\mathcal{G}(M,b)$ the associated group of geometromorphisms. Then the following relation holds true
\begin{equation*}
\Phi^{\star}\{F,G\}_{b}=\{\Phi^{\star}F,\Phi^{\star}G\}_b, \forall F,G\in\mathcal{C}^{\infty}(M,\mathbb{R}), ~\forall \Phi\in\mathcal{G}(M,b),
\end{equation*}
or equivalently, in terms of group actions
\begin{align*}
\tau(\Phi)\bullet\{F,G\}_b =\{\tau(\Phi)\bullet F,\tau(\Phi)\bullet G\}_b,~ \forall F,G\in\mathcal{C}^{\infty}(M,\mathbb{R}), ~\forall \Phi\in\mathcal{G}(M,b),
\end{align*}
where $\tau(\Psi)\bullet H:=H\circ\Psi^{-1},~\forall H\in\mathcal{C}^{\infty}(M,\mathbb{R}), ~\forall \Psi\in\mathcal{G}(M,b)$.
\end{corollary}

Next, we prove that the $b-$bracket on a geometric manifold $(M,b)$, induces two more natural brackets on the manifold $M$, namely the symmetric and the skew--symmetric part of the $b-$bracket, which will be in their turn Leibniz brackets on $M$.

Let us start with the definition of the symmetric bracket associated to the $b-$bracket $\{\cdot,\cdot\}_b$.

\begin{definition}\label{simbr}
Let $b$ be a geometric structure on a manifold $M$, and $\{\cdot,\cdot\}_{b}$ the associated $b-$bracket. Then the application $\{\cdot,\cdot\}_{b,sym}:\mathcal{C}^{\infty}(M,\mathbb{R})\times\mathcal{C}^{\infty}(M,\mathbb{R})\rightarrow \mathcal{C}^{\infty}(M,\mathbb{R})$ given by
$$
\{F,G\}_{b,sym}:=\dfrac{1}{2}\left( \{F,G\}_{b} + \{G,F\}_{b}\right), ~\forall F,G\in\mathcal{C}^{\infty}(M,\mathbb{R}),
$$
is called the \textbf{symmetric $b-$bracket} associated to the geometric structure $b$.
\end{definition}

\begin{remark}
If $b$ is a symmetric geometric structure (i.e., $(M,b)$ is a Riemannian or semi--Riemannian manifold) then $\{\cdot,\cdot\}_{b,sym}=\{\cdot,\cdot\}_{b}$, whereas if $b$ is skew-symmetric (i.e., $(M,b)$ is an almost symplectic manifold) then $\{\cdot,\cdot\}_{b,sym}\equiv 0$.
\end{remark}

In the following result we show that the symmetric $b-$bracket is itself a Leibniz bracket on $M$. Moreover, the associated left and right Leibniz vector fields coincide, and equals the geometric mean of the left and right gradient vector fields generated by the geometric structure $b$.

\begin{proposition}
Let $b$ be a geometric structure on a manifold $M$, and $\{\cdot,\cdot\}_{b,sym}$ the associated symmetric $b-$bracket. Then the following assertions hold:
\begin{itemize}
\item[(i)] $\{\cdot,\cdot\}_{b,sym}$ is a symmetric Leibniz bracket, and hence $(M,\{\cdot,\cdot\}_{b,sym})$ is a symmetric Leibniz manifold,
\item[(ii)] for any arbitrary fixed $F\in\mathcal{C}^{\infty}(M,\mathbb{R})$, the associated left and right Leibniz vector fields coincide, and are both equal to $\dfrac{1}{2}\left(\nabla_{b}^{L}F+\nabla_{b}^{R}F\right)$.
\end{itemize}
\end{proposition}
\begin{proof} 
\begin{itemize}
\item[(i)] This is a direct consequence of the fact $\{\cdot,\cdot\}_{b}$ is a Leibniz bracket on $M$.
\item[(ii)] Let $F\in\mathcal{C}^{\infty}(M,\mathbb{R})$. By the symmetry of the bracket $\{\cdot,\cdot\}_{b,sym}$ we get that the left and right Leibniz vector fields associated to $F$ coincide and are given by
$$
G\in\mathcal{C}^{\infty}(M,\mathbb{R})\mapsto \{G,F\}_{b,sym}=\dfrac{1}{2}(b(\nabla_{b}^{L} G,\nabla_{b}^{L} F)+b(\nabla_{b}^{L} F,\nabla_{b}^{L} G))\in\mathcal{C}^{\infty}(M,\mathbb{R}).
$$
Using the Proposition \ref{fimpo}, the above relation equals to
$$
G\in\mathcal{C}^{\infty}(M,\mathbb{R})\mapsto \{G,F\}_{b,sym}=\dfrac{1}{2}(b(\nabla_{b}^{L} G,\nabla_{b}^{L} F)+b(\nabla_{b}^{R} F,\nabla_{b}^{R} G))\in\mathcal{C}^{\infty}(M,\mathbb{R}),
$$
which becomes
$$
G\in\mathcal{C}^{\infty}(M,\mathbb{R})\mapsto \{G,F\}_{b,sym}=\dfrac{1}{2}(\mathrm{d}G\cdot\nabla_{b}^{L}F+\mathrm{d}G\cdot\nabla_{b}^{R}F)\in\mathcal{C}^{\infty}(M,\mathbb{R}),
$$
or equivalently
$$
G\in\mathcal{C}^{\infty}(M,\mathbb{R})\mapsto \{G,F\}_{b,sym}=\dfrac{1}{2}((\nabla_{b}^{L}F)(G)+(\nabla_{b}^{R}F)(G))\in\mathcal{C}^{\infty}(M,\mathbb{R}),
$$
and hence we get the conclusion.
\end{itemize}
\end{proof}

Note that both Theorem \ref{yuo} as well as Corollary \ref{corimp} remain true if we replace the $b-$brackets by their symmetric parts.
\medskip

Let us now give the definition of the skew--symmetric bracket associated to the $b-$bracket $\{\cdot,\cdot\}_b$.

\begin{definition}
Let $b$ be a geometric structure on a manifold $M$, and $\{\cdot,\cdot\}_{b}$ the associated $b-$bracket. Then the application $\{\cdot,\cdot\}_{b,skew}:\mathcal{C}^{\infty}(M,\mathbb{R})\times\mathcal{C}^{\infty}(M,\mathbb{R})\rightarrow \mathcal{C}^{\infty}(M,\mathbb{R})$ given by
$$
\{F,G\}_{b,skew}:=\dfrac{1}{2}\left( \{F,G\}_{b} - \{G,F\}_{b}\right), ~\forall F,G\in\mathcal{C}^{\infty}(M,\mathbb{R}),
$$
is called the \textbf{skew-symmetric $b-$bracket} associated to the geometric structure $b$.
\end{definition}

\begin{remark}
If $b$ is a skew-symmetric geometric structure (i.e., $(M,b)$ is an almost symplectic manifold) then $\{\cdot,\cdot\}_{b,skew}=\{\cdot,\cdot\}_{b}$, whereas if $b$ is symmetric (i.e., $(M,b)$ is a Riemannian or semi--Riemannian manifold) then $\{\cdot,\cdot\}_{b,skew}\equiv 0$.
\end{remark}

Next, we show that the skew--symmetric $b-$bracket is an almost Poisson bracket on $M$. Moreover, the associated Hamilton--Poisson vector field equals half the difference between the left and the right gradient vector fields generated by the geometric structure $b$.

\begin{proposition}
Let $b$ be a geometric structure on a manifold  $M$, and $\{\cdot,\cdot\}_{b,skew}$ the associated skew-symmetric $b-$bracket. Then
\begin{itemize}
\item[(i)] $\{\cdot,\cdot\}_{b,skew}$ is an almost Poisson bracket (i.e., a skew-symmetric Leibniz bracket), and hence $(M,\{\cdot,\cdot\}_{b,skew})$ is an almost Poisson manifold,
\item[(ii)] for any arbitrary fixed $F\in\mathcal{C}^{\infty}(M,\mathbb{R})$, the associated Hamilton--Poisson vector field, $X_F:=\{\cdot,F\}_{b,skew}$, is given by $\dfrac{1}{2}\left(\nabla_{b}^{L}F-\nabla_{b}^{R}F\right)$.
\end{itemize}
\end{proposition}
\begin{proof} 
\begin{itemize}
\item[(i)] The proof of this item is a direct consequence of the fact that $\{\cdot,\cdot\}_{b}$ is a Leibniz bracket on $M$.
\item[(ii)] Let $F\in\mathcal{C}^{\infty}(M,\mathbb{R})$. From the skew-symmetry of the bracket $\{\cdot,\cdot\}_{b,skew}$ we get that the associated left and right Leibniz vector fields of $F$ are opposite to each other. The Hamilton--Poisson vector field (which is by definition the left--Leibniz vector field) associated to $F$ is given by
$$
G\in\mathcal{C}^{\infty}(M,\mathbb{R})\mapsto \{G,F\}_{b,skew}=\dfrac{1}{2}(b(\nabla_{b}^{L} G,\nabla_{b}^{L} F)-b(\nabla_{b}^{L} F,\nabla_{b}^{L} G))\in\mathcal{C}^{\infty}(M,\mathbb{R}).
$$
Using the Proposition \ref{fimpo}, the above relation equals to
$$
G\in\mathcal{C}^{\infty}(M,\mathbb{R})\mapsto \{G,F\}_{b,skew}=\dfrac{1}{2}(b(\nabla_{b}^{L} G,\nabla_{b}^{L} F)-b(\nabla_{b}^{R} F,\nabla_{b}^{R} G))\in\mathcal{C}^{\infty}(M,\mathbb{R}),
$$
which becomes
$$
G\in\mathcal{C}^{\infty}(M,\mathbb{R})\mapsto \{G,F\}_{b,skew}=\dfrac{1}{2}(\mathrm{d}G\cdot\nabla_{b}^{L}F-\mathrm{d}G\cdot\nabla_{b}^{R}F)\in\mathcal{C}^{\infty}(M,\mathbb{R}),
$$
or equivalently
$$
G\in\mathcal{C}^{\infty}(M,\mathbb{R})\mapsto \{G,F\}_{b,skew}=\dfrac{1}{2}((\nabla_{b}^{L}F)(G)-(\nabla_{b}^{R}F)(G))\in\mathcal{C}^{\infty}(M,\mathbb{R}),
$$
and hence we get the conclusion.
\end{itemize}
\end{proof}

\begin{remark}
In the case when $b$ is a \textbf{constant} geometric structure on $\mathbb{R}^n$ (i.e., for any pair of \textbf{linear} functions $F,G\in\mathcal{C}^{\infty}(\mathbb{R}^{n},\mathbb{R})$, their $b-$bracket, $\{F,G\}_{b}$, is a \textbf{constant} function on $\mathbb{R}^{n}$) then the associated skew-symmetric $b-$bracket is also constant, and hence a Poisson bracket (i.e., a skew-symmetric Leibniz bracket  which verifies Jacobi identity). Consequently we get that $(\mathbb{R}^{n},\{\cdot,\cdot\}_{b,skew})$ is actually a Poisson manifold.
\end{remark}

We end this section by pointing out that, both Theorem \ref{yuo} as well as Corollary \ref{corimp} remain true if we replace the $b-$brackets by their skew--symmetric parts.
\medskip

\section{Laplace--like operators generated by geometric structures}

The aim of this section is to introduce and study Laplace--like operators naturally associated to a geometric structure. For doing this, we shall need besides the geometric structure, the manifold be endowed with a volume form, in order to have a well behaving divergence operator. Hence, the natural environment for our purpose will be a volume manifold $(M,\mu)$ (i.e., a smooth orientable manifold together with a fixed volume form $\mu$) endowed with a geometric structure.

Given a volume manifold $(M,\mu)$ and a vector field $X$, the scalar function $\operatorname{div}_{\mu} X$, uniquely defined by the relation $\mathcal{L}_{X}\mu=(\operatorname{div}_{\mu}X) \mu$, is called the \textit{divergence} of $X$ with respect to the volume form $\mu$, where $\mathcal{L}_{X}$ stands for the Lie derivative along the vector field $X$. As $\mathcal{L}_{X}\mu=\mathrm{d}(\mathbf{i}_{X}\mu)+\mathbf{i}_{X}(\mathrm{d}\mu)$ and $\mathrm{d}\mu=0$, the relation $\mathcal{L}_{X}\mu=(\operatorname{div}_{\mu}X) \mu$ is equivalent to $\mathrm{d}(\mathbf{i}_{X}\mu)=(\operatorname{div}_{\mu}X) \mu$. Recall that $\mathbf{i}_{X}\mu$ denotes the interior product of the vector field $X$ and the volume form $\mu$, i.e., $(\mathbf{i}_{X}\mu)(X_1,\dots,X_{n-1}):=\mu(X,X_1,\dots,X_{n-1}), ~\forall X_1,\dots,X_{n-1}\in\mathfrak{X}(M)$, where $n=\dim M$.

Now we have all necessary ingredients to define the Laplace--like operators associated to a geometric structure on a manifold. These operators extend to manifolds the Laplace--like operators from the linear setting, recently introduced in \cite{TDR}. 
\begin{definition}\label{lap2}
Let $b$ be a geometric structure on a volume manifold $(M,\mu)$, and $U\subseteq M$ an open set. Then for $F\in\mathcal{C}^{\infty}(U,\mathbb{R})$, the left--Laplacian of $F$, is given by $\Delta ^{L}_{b,\mu}F:=\operatorname{div_{\mu}}(\nabla ^{L}_{b}F)$, and similarly, the right--Laplacian of $F$, is given by $\Delta ^{R}_{b,\mu}F:=\operatorname{div_{\mu}}(\nabla ^{R}_{b}F)$.
\end{definition}
Note that $\mathcal{C}^2$ is the sufficient regularity class in order the definition of left/right--Laplacian holds. Nevertheless, we choose to work in the smooth class, for having a unitary approach throughout the article, and also to avoid intricate formulations.

Let us now analyze the relations between the Laplace--like operators associated to two different volume forms on the same manifold. More exactly, given another volume form, $\omega$, on the volume manifold $(M,\mu)$, there exists a function $f\in\mathcal{C}^{\infty}(M,\mathbb{R})$ with $f(m)\neq 0, ~\forall m\in M$, such that $\omega=f\mu$. Then using the formula $\operatorname{div}_{f\mu}(X)=\operatorname{div_{\mu}}(X)+\dfrac{1}{f}~\mathrm{d}f\cdot X$ for the vector fields $X=\nabla^{L/R}_{b}F$, we get that $\Delta ^{L/R}_{b,f\mu}F=\Delta ^{L/R}_{b,\mu}F+\dfrac{1}{f}~\mathrm{d}f\cdot \nabla^{L/R}_{b}F$. 

More precisely, in the hypothesis of Definition \ref{lap2}, we have
\begin{align*}
\Delta ^{L}_{b,f\mu}F=\Delta ^{L}_{b,\mu}F+\dfrac{1}{f}~b(\nabla^{L}_{b}f,\nabla^{L}_{b}F),~~~~
\Delta ^{R}_{b,f\mu}F=\Delta ^{R}_{b,\mu}F+\dfrac{1}{f}~b(\nabla^{R}_{b}F,\nabla^{R}_{b}f),
\end{align*}
or equivalently (cf. Proposition \ref{fimpo})
\begin{align*}
\Delta ^{L}_{b,f\mu}F=\Delta ^{L}_{b,\mu}F+ \dfrac{1}{f}~\{f,F\}_{b},  ~~~~
\Delta ^{R}_{b,f\mu}F=\Delta ^{R}_{b,\mu}F+ \dfrac{1}{f}~\{F,f\}_{b}.
\end{align*}
Hence, we proved the following result.
\begin{proposition}\label{lapul}
Let $b$ be a geometric structure on a volume manifold $(M,\mu)$, and let $\omega$ be another volume form on $M$. Then the following relations hold
\begin{align*}
\Delta ^{L}_{b,\omega}=\Delta ^{L}_{b,\mu}+ \dfrac{1}{f}~\{f,\cdot\}_{b},  ~~~~
\Delta ^{R}_{b,\omega}=\Delta ^{R}_{b,\mu}+ \dfrac{1}{f}~\{\cdot,f\}_{b},
\end{align*}
where $f\in\mathcal{C}^{\infty}(M,\mathbb{R})$ is the nowhere--vanishing function given by $\omega=f\mu$.
\end{proposition}

As we deal with a geometric structure $b$ on a volume manifold $(M,\mu)$ where the volume form $\mu$ is explicitly mentioned by definition, from now on, each time the volume form is clear from the context, we will use the shorthand notation $\Delta ^{L/R}_{b,\mu}=:\Delta ^{L/R}_{b}$.

Next, we present some particular examples of Laplace--like operators, defined by a geometric structure on a Riemannian manifold. The particularity of this case is due to the existence of a natural divergence operator, associated to the canonical volume form induced by the Riemannian metric.
\begin{remark}\label{lap0}
If $b$ is a geometric structure on an orientable Riemannian manifold $(M,g)$ and $U\subseteq M$ an open set, then for $F\in\mathcal{C}^{\infty}(U,\mathbb{R})$, the left--Laplacian of $F$ is given by $\Delta ^{L}_{b}F=\operatorname{div_{g}}(\nabla ^{L}_{b}F)$, while the right--Laplacian of $F$ is given by $\Delta ^{R}_{b}F=\operatorname{div_{g}}(\nabla ^{R}_{b}F)$, where $\operatorname{div}_{g}X$ is the classical divergence operator induced by $g$, i.e., determined uniquely by the relation $\mathcal{L}_{X}\mu_{g}=(\operatorname{div}_{g} X) \mu_{g}$, with $\mu_{g}$ being the canonical Riemannian volume form on $M$ induced by $g$.
\end{remark}

Let us now give a representation of the main quantities from Remark \ref{lap0} in terms of local coordinates, $x=(x^1,\dots,x^n)$, where $n=\operatorname{dim}M$. More precisely, the metric $g$ is locally represented by the symmetric and positive definite matrix
\begin{equation*}
(g_{ij}(x))_{1\leq i,j\leq n}, ~ g_{ij}=g\left(\dfrac{\partial}{\partial x^i},\dfrac{\partial}{\partial x^j}\right),
\end{equation*}
whose determinant is denoted by $|g|:=\det(g_{ij})$.

For any vector field, locally given by $X=X^1 \dfrac{\partial}{\partial x^1}+\dots+X^n \dfrac{\partial}{\partial x^n}$, the expression for $\operatorname{div}_{g}X$ is 
\begin{align}\label{divloc}
\operatorname{div}_{g}X=\dfrac{1}{\sqrt{|g|}}\sum_{i=1}^{n}\dfrac{\partial}{\partial x^i}(\sqrt{|g|}~X^i).
\end{align}
Thus, the above equality together with the expressions \eqref{grdloc} of the left/right--gradient vector fields, imply the following local representations for the $b-$Laplace operators of a $\mathcal{C}^{\infty}$ scalar function $F$:
\begin{align*}
\Delta ^{L}_{b}F=\dfrac{1}{\sqrt{|g|}}\sum_{i=1}^{n}\dfrac{\partial}{\partial x^i}\left(\sum_{j=1}^{n}\mathcal{B}^{ji}\sqrt{|g|}~\dfrac{\partial F}{\partial x^{j}}\right),\\
\Delta ^{R}_{b}F=\dfrac{1}{\sqrt{|g|}}\sum_{i=1}^{n}\dfrac{\partial}{\partial x^i}\left(\sum_{j=1}^{n}\mathcal{B}^{ij}\sqrt{|g|}~\dfrac{\partial F}{\partial x^{j}}\right),
\end{align*}
where the geometric structure $b$ is locally given by the matrix 
\begin{equation*}
\mathcal{B}_x:=(\mathcal{B}_{ij}(x))_{1\leq i,j\leq n}, ~ \mathcal{B}_{ij}=b\left(\dfrac{\partial}{\partial x^i},\dfrac{\partial}{\partial x^j}\right),
\end{equation*}
and $(\mathcal{B}^{ij}(x))_{1\leq i,j\leq n}=\mathcal{B}^{-1}_x$.

Let us return now to the general case of Laplace--like operators on a volume manifold endowed with a geometric structure, and see what happens if the geometric strcuture is symmetric, skew--symmetric or constant.
\begin{remark}\label{REMIQ} The following assertions are direct consequences of Definition \ref{lap2}.
\begin{itemize}
\item[(i)] If $b$ is \textbf{symmetric} (i.e., $b$ is a Riemannian or semi--Riemannian structure) then $\Delta_{b}^{L}=\Delta_{b}^{R}$, thus inducing a geometric $b-$Laplace operator, denoted by $\Delta_b$. Particularly, if $b=g$ is a Riemannian metric on $M$ and $\mu=\mu_{g}$, then $\Delta_{b}=\Delta_{g}$, where $\Delta_g$ is the Laplace--Beltrami operator on the Riemannian manifold $(M,g)$; if $b=h$ with $h$ a Lorentzian metric on $M$ and $\mu=\mu_{h}$, then $\Delta_{b}=\square_h$, where $\square_h$ is the d'Alembert operator on the Lorentzian manifold $(M,h)$.
\item[(ii)]If $b$ is \textbf{skew--symmetric} (i.e., $b$ is an almost symplectic structure) then $\Delta_{b}^{L}=-\Delta_{b}^{R}$. Note that in order to exist an almost symplectic structure on $M$, its dimension must be even ($\operatorname{dim}M=2n$). Recall from Remark \ref{REMI} that \textbf{if moreover} $\mathrm{d}b=0$ (i.e., $(M,b)$ is a symplectic manifold) then for any $H\in\mathcal{C}^{\infty}(M,\mathbb{R})$, the associated Hamiltonian vector field $X_H$ equals $\nabla^{L}_{b}H=-\nabla^{R}_{b}H$. As $\mathcal{L}_{X_H}b=\mathrm{d}(\mathbf{i}_{X_H}b)+\mathbf{i}_{X_H}(\mathrm{d}b)=\mathrm{d}(\mathrm{d}H)+\mathbf{i}_{X_H}(0)=0$, we get that $\mathcal{L}_{X_H}\Lambda=0$ (where $\Lambda:=\dfrac{(-1)^{[n/2]}}{n!}~b^n$, is the Liouville volume form, and $b^n:=\underbrace{b\wedge\dots\wedge b}_\text{n-times}$), and hence $\operatorname{div}_{\Lambda}X_H=0, ~\forall H\in\mathcal{C}^{\infty}(M,\mathbb{R})$. Consequently, for any $H\in\mathcal{C}^{\infty}(M,\mathbb{R})$, $\Delta^{L}_{b,\Lambda}H=\operatorname{div}_{\Lambda}(\nabla^{L}_{b}H)=\operatorname{div}_{\Lambda}X_H=0$, and since $\Delta^{R}_{b,\Lambda}=-\Delta^{L}_{b,\Lambda}$, it follows that also $\Delta^{R}_{b,\Lambda}H=0$. Thus, on a symplectic manifold, \textbf{the Laplace--like operators defined with respect to the Liouville volume form, vanish identically}. On the other hand, by Example \ref{symplbr} and Proposition \ref{lapul}, it follows that if $\omega$ is an arbitrary fixed volume form on the symplectic manifold $(M,b)$, then $\Delta^{R}_{b,\omega}=\dfrac{1}{f}\mathcal{L}_{X_f}=-\Delta^{L}_{b,\omega}$, where $f\in\mathcal{C}^{\infty}(M,\mathbb{R})$ is the nowhere--vanishing function given by $\omega=f\Lambda$. 

\item[(iii)]If $b$ is a \textbf{constant} geometric structure on $\mathbb{R}^n$ endowed with the canonical volume form, it was recently proved in \cite{TDR} that $\Delta_{b}^{L}=\Delta_{b}^{R}$, \textbf{regardless the symmetry--like properties of} $b$.
\end{itemize}
\end{remark}

\subsection{Fundamental properties of Laplace--like operators}

In this subsection we provide some fundamental properties of the Laplace--like operators on a general volume manifold endowed with a geometric structure.
\begin{proposition}\label{pimpi}
Let $b$ be a geometric structure on a volume manifold $(M,\mu)$, and $U\subseteq M$ an open set. Then the following identities hold, for any $F,G\in\mathcal{C}^{\infty}(U,\mathbb{R})$ and $s,t\in\mathbb{R}$
\begin{itemize}
\item[(i)] $\Delta_{b}^{L/R}(s F+t G)=s \Delta_{b}^{L/R} F+t \Delta_{b}^{L/R} G,$
\item[(ii)] $\Delta_{b}^{L/R}(FG)=F\Delta_{b}^{L/R}G+G\Delta_{b}^{L/R}F+b(\nabla_{b}^{L/R}F,\nabla_{b}^{L/R}G)+b(\nabla_{b}^{L/R}G,\nabla_{b}^{L/R}F),$
\item[(iii)] $\Delta_{b}^{L}(FG)-F\Delta_{b}^{L}G-G\Delta_{b}^{L}F=\Delta_{b}^{R}(FG)-F\Delta_{b}^{R}G-G\Delta_{b}^{R}F.$
\end{itemize}
\end{proposition}
\begin{proof}
$(i)$  For any $F,G\in\mathcal{C}^{\infty}(U,\mathbb{R})$ and $s,t\in\mathbb{R}$ we have
\begin{align*}
\Delta_{b}^{L/R}(s F+t G)&=\Delta_{b}^{L/R}(s F+t G)=\operatorname{div_\mu}(\nabla_{b}^{L/R}(s F+t G))=\operatorname{div_\mu}(s\nabla_{b}^{L/R} F+t\nabla_{b}^{L/R}G)\\
&=s\operatorname{div_\mu}(\nabla_{b}^{L/R} F)+t \operatorname{div_\mu}(\nabla_{b}^{L/R} G)=s \Delta_{b}^{L/R} F+t \Delta_{b}^{L/R} G.
\end{align*}
$(ii)$  We prove the required identity only for the left--Laplacian, since for the right--Laplacian the arguments are similar. Thus, for any $F,G\in\mathcal{C}^{\infty}(U,\mathbb{R})$ we get successively
\begin{align*}
\Delta_{b}^{L}(FG)&=\operatorname{div_\mu}(\nabla_{b}^{L}(FG))=\operatorname{div_\mu}(G\nabla_{g}^{L}F)+\operatorname{div_\mu}(F\nabla_{g}^{L}G)=G\operatorname{div_\mu}(\nabla_{b}^{L}F)+(\nabla_{b}^{L}F)(G)\\
&+F\operatorname{div_\mu}(\nabla_{b}^{L}G)+(\nabla_{b}^{L}G)(F)=(\nabla_{b}^{L}F)(G)+(\nabla_{b}^{L}G)(F)+F\operatorname{div_\mu}(\nabla_{b}^{L}G)\\
&+G\operatorname{div_\mu}(\nabla_{b}^{L}F)=\mathrm{d}G\cdot\nabla_{b}^{L}F +\mathrm{d}F\cdot\nabla_{b}^{L}G+F\Delta_{b}^{L}G+G\Delta_{b}^{L}F\\
&=b(\nabla_{b}^{L}G,\nabla_{b}^{L}F)+b(\nabla_{b}^{L}F,\nabla_{b}^{L}G)+F\Delta_{b}^{L}G+G\Delta_{b}^{L}F.
\end{align*}
$(iii)$ Using the previous identity, and the Proposition \ref{fimpo}, it follows that for any $F,G\in\mathcal{C}^{\infty}(U,\mathbb{R})$
\begin{align*}
\Delta_{b}^{L}(FG)-F\Delta_{b}^{L}G-&G\Delta_{b}^{L}F=b(\nabla_{b}^{L}F,\nabla_{b}^{L}G)+b(\nabla_{b}^{L}G,\nabla_{b}^{L}F)\\
&=b(\nabla_{b}^{R}F,\nabla_{b}^{R}G)+b(\nabla_{b}^{R}G,\nabla_{b}^{L}R)=\Delta_{b}^{R}(FG)-F\Delta_{b}^{R}G-G\Delta_{b}^{R}F,
\end{align*}
and thus we obtained the required identity.
\end{proof}

Recall now from the Remark \ref{REMIQ} that on a symplectic manifold $(M,b)$, the Laplace--like operators associated to any arbitrary fixed volume form $\omega$, are given by $\Delta^{R}_{b,\omega}=\dfrac{1}{f}\mathcal{L}_{X_f}=-\Delta^{L}_{b,\omega}$, where $f\in\mathcal{C}^{\infty}(M,\mathbb{R})$ is the nowhere--vanishing function defined by the equality $\omega=f\Lambda$, with $\Lambda$ being the Liouville volume form. Thus, both Laplace--like operators verify the Leibniz identity, since the Lie derivative, $\mathcal{L}_{X_f} : \mathcal{C}^{\infty}(M,\mathbb{R})\rightarrow  \mathcal{C}^{\infty}(M,\mathbb{R})$, do so. 

Obviously, this does not happen generally for Laplace--like operators generated by an arbitrary geometric structure $b$, as it is well known from classical properties of the Laplace--Beltrami operator (if $b$ is a Riemannian metric), or the d'Alembert operator (if $b$ is a Lorentzian metric).

The following result gives an analytic interpretation of the symmetric $b-$bracket, as being a measure for the failure of the Laplace--like operators to verify the Leibniz identity. 

\begin{theorem}\label{BE}
Let $b$ be a geometric structure on a volume manifold $(M,\mu)$, and $\{\cdot,\cdot\}_{b}$, $\{\cdot,\cdot\}_{b,sym}$ the associated $b-$bracket and symmetric $b-$bracket, respectively. Then 
\begin{itemize}
\item[(i)]$\{F,G\}_{b,sym}=\dfrac{1}{2}\left(\Delta_{b}^{L/R}(FG)-F\Delta_{b}^{L/R}G-G\Delta_{b}^{L/R}F\right), ~\forall F,G\in \mathcal{C}^{\infty}(M,\mathbb{R}),$
\item[(ii)] $\Delta_{b}^{L/R}\phi(F)=\phi^{\prime}(F)\Delta_{b}^{L/R}F+\phi^{\prime\prime}(F)\{F,F\}_{b}, ~ \forall F\in\mathcal{C}^{\infty}(M,\mathbb{R}), ~\forall\phi\in\mathcal{C}^{\infty}(\mathbb{R},\mathbb{R}).$
\end{itemize}
\end{theorem}
\begin{proof}
\begin{itemize}
\item[(i)] The proof follows directly from Proposition \ref{pimpi}, Definition \ref{simbr} and Proposition \ref{fimpo}.
\item[(ii)] For any $F\in\mathcal{C}^{\infty}(M,\mathbb{R})$ and $\phi\in\mathcal{C}^{\infty}(\mathbb{R},\mathbb{R})$ we get successively
\begin{align*}
\Delta_{b}^{L/R}\phi(F)&=\operatorname{div_{\mu}}(\nabla ^{L/R}_{b}\phi(F))=\operatorname{div_{\mu}}(\phi^{\prime}(F)\nabla ^{L/R}_{b}F)=\phi^{\prime}(F)\operatorname{div_{\mu}}(\nabla ^{L/R}_{b}F)\\
&+(\nabla_{b}^{L/R}F)(\phi^{\prime}(F))=\phi^{\prime}(F)\Delta_{b}^{L/R}F+\mathrm{d}(\phi^{\prime}(F))\cdot\nabla_{b}^{L/R}F\\
&=\phi^{\prime}(F)\Delta_{b}^{L/R}F+\phi^{\prime\prime}(F)\mathrm{d}F\cdot\nabla_{b}^{L/R}F=\phi^{\prime}(F)\Delta_{b}^{L/R}F\\
&+\phi^{\prime\prime}(F)b(\nabla_{b}^{L/R}F,\nabla_{b}^{L/R}F)=\phi^{\prime}(F)\Delta_{b}^{L/R}F+\phi^{\prime\prime}(F)\{F,F\}_{b}.
\end{align*}
\end{itemize}
\end{proof}

As the $b-$bracket, and hence the symmetric $b-$bracket too, are \textbf{independent of the volume form}, Theorem \ref{BE} implies that \textbf{for any volume form} $\omega$ on the volume manifold $(M,\mu)$, the associated Laplace--like operators verify the identities
\begin{itemize}
\item[(i)]$\{F,G\}_{b,sym}=\dfrac{1}{2}\left(\Delta_{b,\omega}^{L/R}(FG)-F\Delta_{b,\omega}^{L/R}G-G\Delta_{b,\omega}^{L/R}F\right), ~\forall F,G\in \mathcal{C}^{\infty}(M,\mathbb{R}),$
\item[(ii)] $\Delta_{b,\omega}^{L/R}\phi(F)=\phi^{\prime}(F)\Delta_{b,\omega}^{L/R}F+\phi^{\prime\prime}(F)\{F,F\}_{b}, ~ \forall F\in\mathcal{C}^{\infty}(M,\mathbb{R}), ~\forall\phi\in\mathcal{C}^{\infty}(\mathbb{R},\mathbb{R}).$
\end{itemize}
Consequently, the following relations also hold
\begin{itemize}
\item[(i)]$\Delta_{b,\mu}^{L/R}(FG)-F\Delta_{b,\mu}^{L/R}G-G\Delta_{b,\mu}^{L/R}F=\Delta_{b,\omega}^{L/R}(FG)-F\Delta_{b,\omega}^{L/R}G-G\Delta_{b,\omega}^{L/R}F, ~\forall F,G\in \mathcal{C}^{\infty}(M,\mathbb{R}),$
\item[(ii)] $\Delta_{b,\mu}^{L/R}\phi(F)-\phi^{\prime}(F)\Delta_{b,\mu}^{L/R}F=\Delta_{b,\omega}^{L/R}\phi(F)-\phi^{\prime}(F)\Delta_{b,\omega}^{L/R}F, ~ \forall F\in\mathcal{C}^{\infty}(M,\mathbb{R}), ~\forall\phi\in\mathcal{C}^{\infty}(\mathbb{R},\mathbb{R}).$
\end{itemize}

\subsection{Laplace--like operators and geometromorphisms}

This subsection is devoted to analyzing the compatibility between left/right--Laplace operators and geometromorphisms. For doing this, we need to recall first a classical result concerning the compatibility between divergence operators and volume preserving diffeomorphisms. In order to provide a self-contained presentation, we shall give also a proof of this result.

\begin{proposition}\label{divi}
Let $(M,\mu^{M})$, $(N,\mu^{N})$ be two volume manifolds, and $\Phi:M\rightarrow N$ a diffeomorphism such that $\Phi^{\star}\mu^{N}=\mu^M$. Then the following relation holds true
\begin{equation*}
\Phi^{\star}\circ\operatorname{div}_{\mu^N}\circ\Phi_{\star}=\operatorname{div}_{\mu^M},
\end{equation*}
or explicitly
\begin{equation*}
\Phi^{\star}\operatorname{div}_{\mu^N}(\Phi_{\star}X)=\operatorname{div}_{\mu^M}X, ~\forall X\in\mathfrak{X}(M).
\end{equation*}
\end{proposition}
\begin{proof}
For any arbitrary fixed $X\in\mathfrak{X}(M)$, we have
\begin{align*}
(\Phi^{\star}\operatorname{div}_{\mu^N}(\Phi_{\star}X))\mu^M &=(\Phi^{\star}\operatorname{div}_{\mu^N}(\Phi_{\star}X))\Phi^{\star}\mu^N=\Phi^{\star}(\operatorname{div}_{\mu^N}(\Phi_{\star}X)\mu^N )=\Phi^{\star}(\mathrm{d}(\mathbf{i}_{\Phi_{\star}X}\mu^N ))\\
&=\mathrm{d}(\Phi^{\star}(\mathbf{i}_{\Phi_{\star}X}\mu^{N}))=\mathrm{d}(\mathbf{i}_{X}(\Phi^{\star}\mu^N))=\mathrm{d}(\mathbf{i}_{X}\mu^M)=(\operatorname{div}_{\mu^M}X)\mu^M,
\end{align*}
and consequently 
\begin{equation*}
\Phi^{\star}\operatorname{div}_{\mu^N}(\Phi_{\star}X)=\operatorname{div}_{\mu^M}X.
\end{equation*}
\end{proof}

An important particular case of Proposition \ref{divi} occures when  $(M,\mu^{M})=(N,\mu^{N})$. More precisely, the following result holds.

\begin{corollary}
Let $(M,\mu)$ be a volume manifold, and $\operatorname{Diff}_{\mu}(M):=\{\Phi\in\operatorname{Diff}(M)\mid \Phi^{\star}\mu=\mu\}$ the associated group of volume preserving diffeomorphisms. Then
\begin{equation*}
\operatorname{div}_{\mu}\circ\Phi_{\star}=(\Phi^{-1})^{\star}\circ\operatorname{div}_{\mu}, ~\forall \Phi\in\operatorname{Diff}_{\mu}(M),
\end{equation*}
or explicitly, in terms of group actions, 
\begin{align}\label{divequiv}
\operatorname{div}_{\mu}(\tilde{\tau}(\Phi)\bullet X)=\tau(\Phi)\bullet\operatorname{div}_{\mu}X, ~\forall X\in\mathfrak{X}(M),~ \forall \Phi\in\operatorname{Diff}_{\mu}(M),
\end{align}
where the actions of $\operatorname{Diff}(M)$ (and therefore of any of its subgroups, e.g., $\operatorname{Diff}_{\mu}(M)$) on $\mathcal{C}^{\infty}(M,\mathbb{R})$ and $\mathfrak{X}(M)$, are given by
\begin{itemize}
\item[(i)] $\tau(\Phi)\bullet F:=(\Phi^{-1})^{\star}F,~ \forall F\in\mathcal{C}^{\infty}(M,\mathbb{R}), ~\forall \Phi\in\operatorname{Diff}(M),$
\item[(ii)] $\tilde{\tau}(\Phi)\bullet X:=\Phi_{\star}X,~ \forall X\in\mathfrak{X}(M), ~\forall \Phi\in\operatorname{Diff}(M).$
\end{itemize}
The relation \eqref{divequiv} actually says that the operator $\operatorname{div}_{\mu}$ is $\operatorname{Diff}_{\mu}(M)-$equivariant.
\end{corollary}

Next result gives a natural compatibility between left/right--Laplace operators on volume manifolds endowed with geometric structures, and volume preserving geometromorphisms.

\begin{theorem}\label{laplaceinvar}
Let $(M,\mu^{M},b^M)$, $(N,\mu^{N},b^N)$ be two volume manifolds endowed with geometric structures, and $\Phi:M\rightarrow N$ a geometromorphism such that $\Phi^{\star}\mu^{N}=\mu^M$. Then the following relation holds true
\begin{equation*}
\Delta^{L/R}_{b^M}\circ\Phi^{\star}=\Phi^{\star}\circ\Delta^{L/R}_{b^N},
\end{equation*}
or explicitly
\begin{equation*}
\Delta^{L/R}_{b^M}(\Phi^{\star}F)=\Phi^{\star}(\Delta^{L/R}_{b^N}F), ~\forall F\in\mathcal{C}^{\infty}(M,\mathbb{R}).
\end{equation*}
\end{theorem}
\begin{proof}
Using Theorem \ref{grdinvar}, Proposition \ref{divi} and the definition of $b-$Laplace operators it follows that
\begin{align*}
\Delta^{L/R}_{b^M}\circ\Phi^{\star}=\operatorname{div}_{\mu^M}\circ\nabla^{L/R}_{b^M}\circ\Phi^{\star}&=\Phi^{\star}\circ (\Phi^{\star})^{-1}\circ\operatorname{div}_{\mu^M}\circ\Phi_{\star}^{-1}\circ\Phi_{\star}\circ\nabla^{L/R}_{b^M}\circ\Phi^{\star}\\
&=\Phi^{\star}\circ ((\Phi^{\star})^{-1}\circ\operatorname{div}_{\mu^M}\circ\Phi_{\star}^{-1})\circ(\Phi_{\star}\circ\nabla^{L/R}_{b^M}\circ\Phi^{\star})\\
&=\Phi^{\star}\circ\operatorname{div}_{\mu^N}\circ\nabla^{L/R}_{b^N}=\Phi^{\star}\circ\Delta^{L/R}_{b^N},
\end{align*}
and hence we obtained the conclusion.
\end{proof}

Let's see now what Theorem \ref{laplaceinvar} becomes in the particular case when $(M,\mu^{M},b^M)=(N,\mu^{N},b^N)$.

\begin{corollary}
Let $(M,\mu,b)$ be a volume manifold endowed with a geometric structure, and $\operatorname{Diff}_{\mu}(M,b):=\operatorname{Diff}_{\mu}(M)\cap\mathcal{G}(M,b)$ the associated group of volume preserving geometromorphisms. Then
\begin{equation*}
\Delta^{L/R}_{b}\circ\Phi^{\star}=\Phi^{\star}\circ\Delta^{L/R}_{b}, ~\forall \Phi\in\operatorname{Diff}_{\mu}(M,b),
\end{equation*}
or explicitly, in terms of group actions, 
\begin{align*}
\Delta^{L/R}_{b}(\tau(\Phi)\bullet F)=\tau(\Phi)\bullet(\Delta^{L/R}_{b}F), ~\forall F\in\mathcal{C}^{\infty}(M,\mathbb{R}),~ \forall \Phi\in\operatorname{Diff}_{\mu}(M,b),
\end{align*}
i.e., the operators $\Delta^{L}_{b}, \Delta^{R}_{b}$ are both $\operatorname{Diff}_{\mu}(M,b)-$equivariant.
\end{corollary}

\section{Green's identities for the Laplace--like operators}

In this section we provide an equivalent of Green's identities for Laplace--like operators generated by a general geometric structure $b$. In the case when $b$ is a Riemannian metric, we recover the classical Green identities on Riemannian manifolds, while if $b$ is a symplectic structure, we recover some well known identities regarding the integration (with respect to the Liouville volume form) of Poisson bracket of compactly supported smooth functions. 

Let us state now the main result of this section, which gives the Green--like identities for the left/right--Laplace operators on a general geometric manifold.

\begin{theorem}\label{greengen}
Let $M$ be an oriented manifold with boundary, and $\mu$ a volume form on $M$. Let $b$ be a geometric structure on $M$, and $\{\cdot,\cdot\}_{b}$ the associated $b-$bracket. Then for any $F,G\in\mathcal{C}_{c}^{\infty}(M,\mathbb{R})$ the following identities hold:
\begin{itemize}
\item[(i)] 
\begin{equation*}
\int_{M}F\Delta_{b}^{L}G~\mu=-\int_{M}\{F,G\}_{b}~\mu+\int_{\partial M} F~\mathbf{i}_{\nabla_{b}^{L}G}~\mu,
\end{equation*}
\item[(ii)] 
\begin{equation*}
\int_{M}F\Delta_{b}^{R}G~\mu=-\int_{M}\{G,F\}_{b}~\mu+\int_{\partial M} F ~\mathbf{i}_{\nabla_{b}^{R}G}~\mu,
\end{equation*}
\item[(iii)] 
\begin{equation*}
\int_{M}(F\Delta_{b}^{L}G-G\Delta_{b}^{R}F)~\mu=\int_{\partial M}\mathbf{i}_{F\nabla_{b}^{L}G-G\nabla_{b}^{R}F}~\mu.
\end{equation*}
\end{itemize}
\end{theorem}
\begin{proof}
\begin{itemize}
\item[(i)] The proof is based on applying the divergence theorem to the vector field $F\nabla_{b}^{L}G$, defined on the volume manifold $(M,\mu)$. In order to do this, we first compute
\begin{align*}
\operatorname{div_{\mu}}(F\nabla ^{L}_{b}G)&=F\operatorname{div_{\mu}}(\nabla_{b}^{L}G)+(\nabla_{b}^{L}G)(F)=F\Delta_{b}^{L}G+\mathrm{d}F\cdot\nabla_{b}^{L}G\\
&=F\Delta_{b}^{L}G+b(\nabla_{b}^{L}F,\nabla_{b}^{L}G)=F\Delta_{b}^{L}G+\{F,G\}_{b}.
\end{align*}
The above relation together with the divergence theorem yield
\begin{align*}
\int_{M}F\Delta_{b}^{L}G~\mu&=-\int_{M}\{F,G\}_{b}~\mu+\int_{M}\operatorname{div_{\mu}}(F\nabla ^{L}_{b}G)~\mu\\
&=-\int_{M}\{F,G\}_{b}~\mu+\int_{\partial M}\mathbf{i}_{F\nabla_{b}^{L}G} ~\mu\\
&=-\int_{M}\{F,G\}_{b}~\mu+\int_{\partial M}F~\mathbf{i}_{\nabla_{b}^{L}G} ~\mu,
\end{align*}
and thus we deduce the required identity.
\item[(ii)] For this item we use the divergence theorem for the vector field $F\nabla_{b}^{R}G$. As before, we first compute 
\begin{align*}
\operatorname{div_{\mu}}(F\nabla ^{R}_{b}G)&=F\operatorname{div_{\mu}}(\nabla_{b}^{R}G)+(\nabla_{b}^{R}G)(F)=F\Delta_{b}^{R}G+\mathrm{d}F\cdot\nabla_{b}^{R}G\\
&=F\Delta_{b}^{R}G+b(\nabla_{b}^{R}G,\nabla_{b}^{R}F)=F\Delta_{b}^{R}G+\{G,F\}_{b}.
\end{align*}
Applying the divergence theorem we obtain
\begin{align*}
\int_{M}F\Delta_{b}^{R}G~\mu&=-\int_{M}\{G,F\}_{b}~\mu+\int_{M}\operatorname{div_{\mu}}(F\nabla ^{R}_{b}G)~\mu\\
&=-\int_{M}\{G,F\}_{b}~\mu+\int_{\partial M} \mathbf{i}_{F\nabla_{b}^{R}G}~\mu\\
&=-\int_{M}\{G,F\}_{b}~\mu+\int_{\partial M} F~\mathbf{i}_{\nabla_{b}^{R}G}~\mu.
\end{align*}
\item[(iii)] Using relations $(i)$ and $(ii)$, we successively get
\begin{align*}
\int_{M}(F\Delta_{b}^{L}G-G\Delta_{b}^{R}F)~\mu &=\int_{M}F\Delta_{b}^{L}G~\mu-\int_{M}G\Delta_{b}^{R}F~\mu=-\int_{M}\{F,G\}_{b}~\mu\\
&+\int_{\partial M} F~\mathbf{i}_{\nabla_{b}^{L}G}~\mu -\left(-\int_{M}\{F,G\}_{b}~\mu+\int_{\partial M} G~\mathbf{i}_{\nabla_{b}^{R}F}~\mu\right)\\
&=\int_{\partial M}\left( F~\mathbf{i}_{\nabla_{b}^{L}G}~\mu - G~\mathbf{i}_{\nabla_{b}^{R}F}~\mu \right)=\int_{\partial M} \mathbf{i}_{F\nabla_{b}^{L}G-G\nabla_{b}^{R}F}~\mu .
\end{align*}
\end{itemize}
\end{proof}

Before stating next result, let us recall from Riesz representation theorem that given an oriented manifold $M$ together with a fixed volume form $\mu$, there is a unique measure $m_\mu$ on the Borel $\sigma-$algebra of $M$, such that for every continuous and compactly supported scalar function $F\in\mathcal{C}_{c}(M,\mathbb{R})$,
$$
\int_{M}F\mathrm{d}m_{\mu}=\int_{M}F\mu.
$$
For detalis regarding the above mentioned remark, see e.g., \cite{AMR}.

The following result is a restatement of Theorem \ref{greengen} in the case when the manifold $M$ is boundaryless. 
\begin{theorem}\label{green2}
Let $M$ be a boundaryless oriented manifold, and $\mu$ a volume form on $M$. Let $b$ be a geometric structure on $M$, and $\{\cdot,\cdot\}_{b}$ the associated $b-$bracket. Then for any $F,G\in\mathcal{C}_{c}^{\infty}(M,\mathbb{R})$ the following identities hold:
\begin{itemize}
\item[(i)] 
\begin{equation*}
\int_{M}F\Delta_{b}^{L}G~\mathrm{d}m_{\mu}=-\int_{M}\{F,G\}_{b}~\mathrm{d}m_{\mu},
\end{equation*}
\item[(ii)] 
\begin{equation*}
\int_{M}F\Delta_{b}^{R}G~\mathrm{d}m_{\mu}=-\int_{M}\{G,F\}_{b}~\mathrm{d}m_{\mu},
\end{equation*}
\item[(iii)] 
\begin{equation*}
(\Delta_{b}^{L}G,F)=\int_{M}F\Delta_{b}^{L}G~\mathrm{d}m_{\mu}=\int_{M}G\Delta_{b}^{R}F~\mathrm{d}m_{\mu}=( G,\Delta_{b}^{R}F),
\end{equation*}
\end{itemize}
where $(\cdot,\cdot)$ denotes the $L^{2}(M,m_{\mu})$ inner product.
\end{theorem}
Obviously, in the case when $M$ is a manifold with boundary, the identities from Theorem \ref{green2} hold also for any compactly supported smooth functions vanishing on $\partial M$.

\begin{remark} Note that the identities from Theorem \ref{greengen} and \ref{green2} still hold true if \textbf{at least one} of the smooth functions $F,G$ is compactly supported. Consequently:
\begin{itemize}
\item[(i)] if $F\equiv 1$, the identities given by Theorem \ref{greengen} reduce to
\begin{equation*}
\int_{M}\Delta_{b}^{L/R}G~\mu=\int_{\partial M} \mathbf{i}_{\nabla_{b}^{L/R}G}~\mu, ~\forall G\in\mathcal{C}^{\infty}_{c}(M,\mathbb{R}),
\end{equation*}
\item[(ii)] if $F\equiv 1$, the identities given by Theorem \ref{green2} reduce to
\begin{equation*}
\int_{M}\Delta_{b}^{L/R}G~\mathrm{d}m_\mu=0, ~\forall G\in\mathcal{C}^{\infty}_{c}(M,\mathbb{R}).
\end{equation*}
\end{itemize}
\end{remark}

\subsection{Green's identities for Laplace--like operators on Riemannian manifolds}

In this subsection we present Green's identities for the left/right--Laplace operators associated to a geometric structure $b$, defined on a Riemannian manifold $(M,g)$. A special case occurs for $b=g$, when we recover the classical Green's identities for the Laplace--Beltrami operator. Before stating the main result, let us recall that if $(M,g)$ is an oriented Riemannian manifold with boundary $\partial M$, then there is a unique outward-pointing unit normal vector field $\nu$ along $\partial M$. In this case the Theorem  \ref{greengen} becomes:

\begin{corollary}\label{greengenR}
Let $(M,g)$ be an oriented Riemannian manifold with boundary $\partial M$, and $\mu_M$, $\mu_{\partial M}$ the canonical Riemannian volume forms on $M$ and $\partial M$, respectively. Let $b$ be a geometric structure on $M$, and $\{\cdot,\cdot\}_{b}$ the associated $b-$bracket. Then for any $F,G\in\mathcal{C}_{c}^{\infty}(M,\mathbb{R})$, the following identities hold:

\begin{itemize}
\item[(i)] 
\begin{equation*}
\int_{M}F\Delta_{b}^{L}G~\mathrm{d}m_{\mu_M}=-\int_{M}\{F,G\}_{b}~\mathrm{d}m_{\mu_M}+\int_{\partial M}F g(\nabla_{b}^{L}G,\nu)~\mathrm{d}m_{\mu_{\partial M}},
\end{equation*}
\item[(ii)] 
\begin{equation*}
\int_{M}F\Delta_{b}^{R}G~\mathrm{d}m_{\mu_M}=-\int_{M}\{G,F\}_{b}~\mathrm{d}m_{\mu_M}+\int_{\partial M}F g(\nabla_{b}^{R}G,\nu)~\mathrm{d}m_{\mu_{\partial M}},
\end{equation*}
\item[(iii)] 
\begin{equation*}
\int_{M}(F\Delta_{b}^{L}G-G\Delta_{b}^{R}F)~\mathrm{d}m_{\mu_M}=\int_{\partial M}g(F\nabla_{b}^{L}G-G\nabla_{b}^{R}F,\nu)~\mathrm{d}m_{\mu_{\partial M}},
\end{equation*}
\end{itemize}
where $\nu$ is the outward--pointing unit normal vector field along $\partial M$.

For $b=g$ we obtain that $\Delta_{b}^{L}=\Delta_{b}^{R}=\Delta_g$ (where $\Delta_g$ is the classical Laplace--Beltrami operator), thus the above relations become the Green's identities on a Riemannian manifold.
\end{corollary}

\begin{remark} Note that as in the general case, the identities from Corollary \ref{greengenR} still hold true if \textbf{at least one} of the smooth functions $F,G$ is compactly supported. Consequently, if $F\equiv 1$ the identities given by Corollary \ref{greengenR} reduce to
\begin{equation*}
\int_{M}\Delta_{b}^{L/R}G~\mathrm{d}m_{\mu_M}=\int_{\partial M}g(\nabla_{b}^{L/R}G,\nu)~\mathrm{d}m_{\mu_{\partial M}}, ~\forall G\in\mathcal{C}^{\infty}_{c}(M,\mathbb{R}).
\end{equation*}
\end{remark}

\subsection{Green's identities for Laplace--like operators on symplectic manifolds}

In this subsection we provide Green's identities for the left/right--Laplace operators on a symplectic manifold. Before stating the result, let us recall from Remark \ref{REMIQ} that if $(M,b)$ is a $2n-$dimensional symplectic manifold manifold, then for any $H\in\mathcal{C}^{\infty}(M,\mathbb{R})$, the associated Hamiltonian vector field $X_H$ equals $\nabla^{L}_{b}H=-\nabla^{R}_{b}H$. Moreover, if $\Lambda=\dfrac{(-1)^{[n/2]}}{n!}~b^n$ denotes the Liouville volume form of $(M,b)$, then $\operatorname{div}_{\Lambda}X_H=0$, for any $H\in\mathcal{C}^{\infty}(M,\mathbb{R})$. Consequently, the $b-$Laplace operators \textbf{associated to $\Lambda$ vanish identically}, i.e., $ \Delta^{R}_{b,\Lambda}H=\Delta^{L}_{b,\Lambda}H\equiv 0, ~\forall H\in\mathcal{C}^{\infty}(M,\mathbb{R})$. If instead $\omega$ is an arbitrary fixed volume form on $M$, then the associated Laplace--like operators are given by $\Delta^{R}_{b,\omega}=\dfrac{1}{f}\mathcal{L}_{X_f}=-\Delta^{L}_{b,\omega}$, where $f\in\mathcal{C}^{\infty}(M,\mathbb{R})$ denotes the nowhere--vanishing function given by $\omega=f\Lambda$. 

Thus, the above remarks together with the skew--symmetry of the $b-$bracket, imply that Theorem \ref{greengen} becomes:

\begin{corollary}\label{corol}
Let $(M,b)$ be a symplectic manifold, and $\{\cdot,\cdot\}_{b}$ the Poisson bracket associated to the symplectic structure $b$. If $\omega$ is an arbitrary given volume form on $M$, then for any $F,G\in\mathcal{C}_{c}^{\infty}(M,\mathbb{R})$ 
\begin{equation*}
\int_{M}\{F,G\}_{b}~\omega=-\int_{M}\dfrac{F}{f}\{f,G\}_{b}~\omega+\int_{\partial M} F ~\mathbf{i}_{X_G}\omega=-\int_{M}\{F,f\}_{b}\dfrac{G}{f}~\omega-\int_{\partial M} G ~\mathbf{i}_{X_F}\omega,
\end{equation*}
where $f\in\mathcal{C}^{\infty}(M,\mathbb{R})$ is the nowhere--vanishing function given by $\omega=f\Lambda$, with $\Lambda$ being the Liouville volume form of the symplectic manifold $(M,b)$.
\end{corollary}
Note that as in the general case, the above identities still hold true if \textbf{at least one} of the smooth functions $F,G$ is compactly supported. Consequently, if $F\equiv 1$ the identities from Corollary \ref{corol} reduce to
\begin{equation*}
\int_{M}\dfrac{1}{f}\{f,G\}_{b}~\omega=\int_{\partial M} \mathbf{i}_{X_G}\omega, ~\forall G\in\mathcal{C}^{\infty}_{c}(M,\mathbb{R}).
\end{equation*}
Let us point out an important particular case of Corollary \ref{corol}. More precisely, if $\omega=\Lambda$, then $\{\cdot,f\}_{b}=\{f,\cdot\}_{b}\equiv 0$ (as $f\equiv 1$), and hence Corollary \ref{corol} becomes the well known result:

\begin{corollary}[\cite{MR}]
Let $(M,b)$ be a symplectic manifold, and $\{\cdot,\cdot\}_{b}$ the Poisson bracket associated to the symplectic structure $b$. If $\Lambda$ is the Liouville volume form of $(M,b)$, then for any $F,G\in\mathcal{C}_{c}^{\infty}(M,\mathbb{R})$ 
\begin{equation*}
\int_{M}\{F,G\}_{b}~\Lambda=\int_{\partial M} F~\mathbf{i}_{X_G}\Lambda =-\int_{\partial M} G~\mathbf{i}_{X_F}\Lambda.
\end{equation*}
If $F|_{\partial M}=0$ or $G|_{\partial M}=0$ then
\begin{equation}\label{ulull}
\int_{M}\{F,G\}_{b}~\Lambda=0.
\end{equation}
Obviously, if $\partial M=\emptyset$, then the relation \eqref{ulull} holds true for any $F,G\in\mathcal{C}_{c}^{\infty}(M,\mathbb{R})$.
\end{corollary}

\subsection{The Dirichlet energy associated to a geometric structure}

The aim of this short subsection is to present the Euler--Lagrange equations associated to Dirichlet energy in the context of a general geometric structure. In the case when the geometric structure is a Riemannian metric, we recover the classical Dirichlet energy, whose Euler--Lagrange equation is precisely the Laplace equation associated to Laplace--Beltrami operator while if the geometric structure is a Lorentzian metric, the Euler--Lagrange equation corresponding the Dirichlet's energy is exactly the wave equation associated to the d'Alembert operator.

As we shall analyze the case of a general geometric structure, in order to define a non--trivial Dirichet energy functional, the geometric structure it must not be skew--symmetric. More precisely, the following result holds. 

\begin{theorem}\label{Diri}
Let $M$ be a boundaryless oriented manifold, and $\mu$ a volume form on $M$. Let $b$ be a geometric structure on $M$ which is not skew--symmetric, and $\{\cdot,\cdot\}_{b}$ the associated $b-$bracket. For $F\in\mathcal{C}^{\infty}(M,\mathbb{R})$ we define the associated Dirichlet energy by
\begin{equation*}
\mathcal{E}(F):=\dfrac{1}{2}\int_{M}\{F,F\}_{b}~\mathrm{d}m_{\mu}.
\end{equation*}
Then the Euler--Lagrange equation for $\mathcal{E}$ is
\begin{equation*}
\dfrac{1}{2}(\Delta^{L}_{b}F+\Delta^{R}_{b}F)=0.
\end{equation*}
\end{theorem}
\begin{proof}
Let $F\in\mathcal{C}^{\infty}(M,\mathbb{R})$ be a critical point of $\mathcal{E}$. Then 
\begin{equation}\label{ELe}
\dfrac{\mathrm{d}}{\mathrm{d}\lambda}\mathcal{E}(F+\lambda\delta F)\mid_{\lambda=0}=0, ~\forall \delta F\in\mathcal{C}_{c}^{\infty}(M,\mathbb{R}).
\end{equation}
Using the properties of the $b-$bracket, for any $\delta F\in\mathcal{C}_{c}^{\infty}(M,\mathbb{R})$ and any $\lambda\in\mathbb{R}$, we get
\begin{align*}
\mathcal{E}(F+\lambda\delta F)&=\dfrac{1}{2}\int_{M}\{F+\lambda\delta F,F+\lambda\delta F\}_{b}~\mathrm{d}m_{\mu}=\dfrac{1}{2}\int_{M}\{F,F\}_{b}~\mathrm{d}m_{\mu}\\
&+\lambda\left(\dfrac{1}{2}\int_{M}\{F,\delta F\}_{b}~\mathrm{d}m_{\mu}+\dfrac{1}{2}\int_{M}\{\delta F,F\}_{b}~\mathrm{d}m_{\mu}\right)+\dfrac{\lambda^2}{2}\int_{M}\{\delta F,\delta F\}_{b}~\mathrm{d}m_{\mu}.
\end{align*}
From the above relations together with Theorem \ref{green2} we obtain that
\begin{align*}
\dfrac{\mathrm{d}}{\mathrm{d}\lambda}\mathcal{E}(F+\lambda\delta F)\mid_{\lambda=0}&=\dfrac{1}{2}\int_{M}\{F,\delta F\}_{b}~\mathrm{d}m_{\mu}+\dfrac{1}{2}\int_{M}\{\delta F,F\}_{b}~\mathrm{d}m_{\mu}\\
&=-\dfrac{1}{2}\int_{M}\Delta^{R}_{b}F \cdot\delta F~\mathrm{d}m_{\mu} - \dfrac{1}{2}\int_{M}\Delta^{L}_{b}F \cdot\delta F~\mathrm{d}m_{\mu}\\
&= - \int_{M}\dfrac{1}{2}\left(\Delta^{L}_{b}F+\Delta^{R}_{b}F\right)\delta F~\mathrm{d}m_{\mu}.
\end{align*}
Consequently, the equality \eqref{ELe} is equivalent to
\begin{equation*}
\int_{M}\dfrac{1}{2}\left(\Delta^{L}_{b}F+\Delta^{R}_{b}F\right)\delta F~\mathrm{d}m_{\mu}=0, ~\forall \delta F\in\mathcal{C}_{c}^{\infty}(M,\mathbb{R}),
\end{equation*}
and thus we get
\begin{equation*}
\dfrac{1}{2}(\Delta^{L}_{b}F+\Delta^{R}_{b}F)=0.
\end{equation*}
\end{proof}

In the hypothesis of Theorem \ref{Diri}, if $b$ is a symmetric geometric structure (i.e., $(M,b)$ is a Riemannian or semi--Riemannian manifold) then $\Delta^{L}_{b}=\Delta^{R}_{b}=:\Delta_b$ and consequently the Euler--Lagrange equations for $\mathcal{E}$ is $\Delta_{b}F=0$. In this case the solutions to Euler--Lagrange equations are the $b-$harmonic functions. 

Notice that if $(M,b)$ is a Riemannian manifold, the equation $\Delta_{b}F=0$ is precisely the classical Laplace equation associated to Laplace--Beltrami operator, while if $(M,b)$ is a Lorentzian manifold, the equation $\Delta_{b}F=0$ is precisely the wave equation associated to d'Alembert operator, i.e., $\square_{b}F=0$.

\section{Dynamical properties of gradient--like vector fields}

In this section we present some important dynamical properties of left/right--gradient vector fields associated to a general geometric structure on a manifold. For particular examples of geometric structures (e.g., Riemannian or semi--Riemannian metrics, symplectic structures) we recover some well known properties of various classes of vector fields (e.g., gradient vector fields, Hamiltonian vector fields).

\subsection{The flow of left/right--gradients and the $b-$bracket}

Let us start by providing a compatibility result between the flow of a left/right--gradient vector field generated by a general geometric structure $b$, and the associated $b-$bracket. This is the correspondent on geometric manifolds of the similar result regarding the Hamiltonian vector fields on symplectic manifolds, see, e.g., \cite{MR}, \cite{RTC}.
\begin{theorem}
Let $b$ be a geometric structure on a manifold $M$, and $\{\cdot,\cdot\}_b$ the associated $b-$bracket. Let $F\in\mathcal{C}^{\infty}(M,\mathbb{R})$ be a smooth function, and $\nabla^{L/R}_{b}F\in\mathfrak{X}(M)$ the corresponding gradients. If $\Phi^{L/R}_{t}$ denotes the flow of $\nabla^{L/R}_{b}F$, then for any $f\in\mathcal{C}^{\infty}(M,\mathbb{R})$
\begin{align*}
\dfrac{\mathrm{d}}{\mathrm{d}t}(\Phi^{L}_{t})^{\star}f=(\Phi^{L}_{t})^{\star}\{f,F\}_b, ~~~~\dfrac{\mathrm{d}}{\mathrm{d}t}(\Phi^{R}_{t})^{\star}f=(\Phi^{R}_{t})^{\star}\{F,f\}_b,
\end{align*}
where $(\Phi^{L/R}_{t})^{\star}g:=g\circ\Phi^{L/R}_{t}$, for any $g\in\mathcal{C}^{\infty}(M,\mathbb{R})$.
\end{theorem}
\begin{proof}
We shall prove first the equality concerning the flow of the left-gradient. In order to do this, for any arbitrary fixed $f\in\mathcal{C}^{\infty}(M,\mathbb{R})$ we have
\begin{align*}
\dfrac{\mathrm{d}}{\mathrm{d}t}(\Phi^{L}_{t})^{\star}f&=(\Phi^{L}_{t})^{\star}(\mathrm{d}f\cdot\nabla^{L}_{b}F)=(\Phi^{L}_{t})^{\star}(b(\nabla^{L}_{b}f,\nabla^{L}_{b}F))\\
&=(\Phi^{L}_{t})^{\star}\{f,F\}_b.
\end{align*}
The equality regarding the flow of the right-gradient follows mimetically. More precisely, for any arbitrary fixed $f\in\mathcal{C}^{\infty}(M,\mathbb{R})$ we have
\begin{align*}
\dfrac{\mathrm{d}}{\mathrm{d}t}(\Phi^{R}_{t})^{\star}f&=(\Phi^{R}_{t})^{\star}(\mathrm{d}f\cdot\nabla^{R}_{b}F)=(\Phi^{R}_{t})^{\star}(b(\nabla^{R}_{b}F,\nabla^{R}_{b}f))\\
&=(\Phi^{R}_{t})^{\star}\{F,f\}_b,
\end{align*}
and hence we get the conclusion.
\end{proof}

\begin{corollary}
Let $b$ be a geometric structure on a manifold $M$, $\{\cdot,\cdot\}_b$ the associated $b-$bracket, and $F,G\in\mathcal{C}^{\infty}(M,\mathbb{R})$. Then the following three assertions are equivalent:
\begin{itemize}
\item[(i)] $\{F,G\}_b =0$,
\item[(ii)] $F$ is a constant of motion of $\nabla^{L}_{b}G$,
\item[(ii)] $G$ is a constant of motion of $\nabla^{R}_{b}F$.
\end{itemize}
\end{corollary}
In the case when $b$ is symmetric (i.e., $(M,b)$ is a Riemannian or a semi--Riemannian manifold) or skew--symmetric (i.e., $(M,b)$ is an almost--symplectic manifold), we recover the well known equivalences (see, e.g., \cite{MR}, \cite{RTC}):
\begin{itemize}
\item[(i)] $\{F,G\}_b =0$,
\item[(ii)] $F$ is a constant of motion of $\nabla_{b}G$, if $b$ is symmetric  ($F$ is a constant of motion of $X_G$, if $b$ is skew--symmetric),
\item[(ii)] $G$ is a constant of motion of $\nabla_{b}F$, if $b$ is symmetric  ($G$ is a constant of motion of $X_F$, if $b$ is skew--symmetric),
\end{itemize}
where $\nabla^{L}_{b}=\nabla^{R}_{b}=:\nabla_{b}$, if $b$ is symmetric, while $\nabla^{L}_{b}=-\nabla^{R}_{b}=:X_{\bullet}=\{\cdot,\bullet\}_b$, if $b$ is skew--symmetric.

\subsection{A transport theorem}

In this short subsection we present a transport theorem associated to flows of left/right--gradient vector fields generated by a general geometric structure. This result gives also a dynamical interpretation of the Laplace--like operators.
\begin{theorem}\label{flowth}
Let $b$ be a geometric structure on an oriented manifold $M$ endowed with a volume form $\mu$, and $F\in\mathcal{C}^{\infty}(M,\mathbb{R})$. If $\Phi_{t}^{L/R}$ denotes the flow of the left/right--gradient vector field $\nabla^{L/R}_{b}F$, then
\begin{equation*}
\dfrac{\mathrm{d}}{\mathrm{d}t} m_{\mu}\left(\Phi_{t}^{L/R}(U)\right)=\int_{\Phi_{t}^{L/R}(U)}\Delta^{L/R}_{b}F ~\mathrm{d}m_{\mu},
\end{equation*}
where $U$ is a compact regular domain contained in the domain of definition of the flow.
\end{theorem}
\begin{proof} For any compact regular domain $U$ contained in the domain of definition of the flow, we have
\begin{align*}
\dfrac{\mathrm{d}}{\mathrm{d}t} m_{\mu}\left(\Phi_{t}^{L/R}(U)\right)&=\dfrac{\mathrm{d}}{\mathrm{d}t}\int_{\Phi_{t}^{L/R}(U)}~\mathrm{d} m_{\mu}=\dfrac{\mathrm{d}}{\mathrm{d}t}\int_{U}~(\Phi_{t}^{L/R})^{\star}\mathrm{d} m_{\mu}=\dfrac{\mathrm{d}}{\mathrm{d}t}\int_{U}~(\Phi_{t}^{L/R})^{\star}\mu\\
&=\int_{U}~\dfrac{\mathrm{d}}{\mathrm{d}t}(\Phi_{t}^{L/R})^{\star}\mu=\int_{U}~(\Phi_{t}^{L/R})^{\star}\mathcal{L}_{\nabla^{L/R}_b F}\mu=\int_{U}~(\Phi_{t}^{L/R})^{\star}\operatorname{div}_{\mu}(\nabla^{L/R}_b F)\mu\\
&=\int_{U}~(\Phi_{t}^{L/R})^{\star}\Delta_{b}^{L/R} F~\mu=\int_{U}~(\Phi_{t}^{L/R})^{\star}\Delta_{b}^{L/R} F~\mathrm{d}m_{\mu}\\
&=\int_{\Phi_{t}^{L/R}(U)}\Delta^{L/R}_{b}F ~\mathrm{d}m_{\mu},
\end{align*} 
and thus we obtained the conclusion.
\end{proof}

\begin{corollary}In the hypothesis of Theorem \ref{flowth}, the following assertions hold.
\begin{itemize}
\item[(i)] If $F$ is left/right--$b-$subharmonic (i.e., $\Delta^{L/R}_{b}F \geq 0$ on $M$) then the function $t\mapsto  m_{\mu}\left(\Phi_{t}^{L/R}(U)\right)$ is increasing. Consequently, as $\Phi_{0}^{L/R}(U)=U$, we get that $m_{\mu}\left(\Phi_{t}^{L/R}(U)\right)\geq m_{\mu}(U)$, for $t\geq 0$. 
\item[(ii)]  If $F$ is let/right--$b-$harmonic (i.e., $\Delta^{L/R}_{b}F = 0$ on $M$) then the function $t\mapsto  m_{\mu}\left(\Phi_{t}^{L/R}(U)\right)$ is constant, and hence we obtain that $m_{\mu}\left(\Phi_{t}^{L/R}(U)\right)= m_{\mu}(U)$, for $t\geq 0$. 

(Recall from Remark \ref{REMIQ} that if $b$ is a symplectic form on $M$, then the $b-$Laplace operators \textbf{defined with respect to the Liouville volume form} $\mu=\Lambda:=\dfrac{(-1)^{[n/2]}}{n!}~b^n$ (where $n=\dfrac{1}{2}\operatorname{dim}M$) \textbf{vanish identically}, i.e., $\Delta^{L/R}_{b}F = 0$ on $M$, \textbf{for any} $F\in\mathcal{C}^{\infty}(M,\mathbb{R})$. Thus, in this case, \textbf{any} $F\in\mathcal{C}^{\infty}(M,\mathbb{R})$ is left/right--$b-$harmonic.)
\item[(iii)]  If $F$ is left/right--$b-$superharmonic (i.e., $\Delta^{L/R}_{b}F \leq  0$ on $M$) then the function $t\mapsto  m_{\mu}\left(\Phi_{t}^{L/R}(U)\right)$ is decreasing, and so it follows that $m_{\mu}\left(\Phi_{t}^{L/R}(U)\right)\leq m_{\mu}(U)$, for $t\geq 0$.
\end{itemize}
\end{corollary}

\subsection{Periodic orbits of left/right--gradient vector fields}

The aim of this subsection is to provide a result concerning the non-existence of periodic orbits of left/right--gradient vector fields, extending to gradient--like vector fields, the well known result for classical gradients. In order to do this we need to introduce first some terminology. 
\begin{definition}
Let $b$ be a geometric structure on a manifold $M$, and $\{\cdot,\cdot\}_{b}$ the associated $b-$bracket. Then $\{\cdot,\cdot\}_{b}$ is called \textbf{positive definite} if $\{F,F\}_{b}(m)>0$, $\forall F\in\mathcal{C}^{\infty}(M,\mathbb{R})$, $\forall m \in M\setminus\operatorname{Crit}(F)$, where $\operatorname{Crit}(F)=\{m\in M \mid \mathrm{d}F_{m}=0\}$. Similarly, $\{\cdot,\cdot\}_{b}$ is called \textbf{negative definite} if $\{F,F\}_{b}(m)<0$, $\forall F\in\mathcal{C}^{\infty}(M,\mathbb{R})$, $\forall m \in M\setminus\operatorname{Crit}(F)$.
\end{definition}
Recall from Definition \ref{simbr} that $\{F,F\}_{b}=\{F,F\}_{b,sym},~\forall F\in\mathcal{C}^{\infty}(M,\mathbb{R})$. Consequently, the $b-$bracket is positive/negative definite if and only if the associated symmetric $b-$bracket is positive/negative definite.

Next, we give the definition of a positive/negative geometric structure.
\begin{definition}
Let $b$ be a geometric structure on a manifold $M$. Then $b$ is called \textbf{positive definite} if $(b(X,X))(m)>0$, for all vector fields $X$ and $m\in M$ such that $X(m)\neq 0$. Similarly, $b$ is called \textbf{negative definite} if $(b(X,X))(m)<0$, for all vector fields $X$ and $m\in M$ such that $X(m)\neq 0$
\end{definition}
Let us recall that given a geometric structure $b$ on a manifold $M$, its symmetric part, $b_{sym}$, is given by $b_{sym}:=\frac{1}{2}(b+b^{op})$, where $b^{op}:\mathfrak{X}(M)\times \mathfrak{X}(M)\rightarrow\mathcal{C}^{\infty}(M,\mathbb{R})$ is the non-degenerate $(0,2)-$tensor field given by $b^{op}(X,Y):=b(Y,X), ~\forall X,Y\in\mathfrak{X}(M)$. Note that if $b_{sym}$ is non-degenerate (i.e., $b_{sym}$ is also a geometric structure on $M$), then $b$ is a positive/negative definite geometric structure if and only if $b_{sym}$ is so. 
\begin{remark}
If $b$ is a positive/negative geometric structure on a manifold $M$, then the associated $b-$braket is positive/negative definite, since $$\{F,F\}_{b}=b(\nabla^{L/R}_{b}F,\nabla^{L/R}_{b}F), ~\forall F\in\mathcal{C}^{\infty}(M,\mathbb{R}).$$
\end{remark}
Let us now state the main result of this subsection.
\begin{theorem}\label{theor1}
Let $b$ be a geometric structure on a manifold $M$, such that the associated $b-$bracket is positive or negative definite. Then for any $F\in\mathcal{C}^{\infty}(M,\mathbb{R})$, non of the vector fields $\nabla_{b}^{L}F$ and $\nabla_{b}^{R}F$ admit non--trivial periodic orbits.
\end{theorem}
\begin{proof}
Suppose that $\gamma:[0,T]\rightarrow M$ is a non--trivial $T-$periodic solution (i.e., $\gamma$ is not a constant periodic function) of the vector field $\nabla_{b}^{L/R}F$. Then we get
\begin{align*}
0&=F(\gamma(T))-F(\gamma(0))=\int_{0}^{T}\dfrac{\mathrm{d}}{\mathrm{d}t}F(\gamma(t))\mathrm{d}t=\int_{0}^{T}\mathrm{d}F_{\gamma(t)}\cdot \dot{\gamma}(t)\mathrm{d}t\\
&=\int_{0}^{T}\mathrm{d}F_{\gamma(t)}\cdot \nabla_{b}^{L/R}F(\gamma(t))\mathrm{d}t=\int_{0}^{T}b_{\gamma(t)}(\nabla_{b}^{L/R}F(\gamma(t)),\nabla_{b}^{L/R}F(\gamma(t)))\mathrm{d}t\\
&=\int_{0}^{T}(b(\nabla_{b}^{L/R}F,\nabla_{b}^{L/R}F))(\gamma(t))\mathrm{d}t=\int_{0}^{T}\{F,F\}_{b}(\gamma(t))\mathrm{d}t,
\end{align*}
which contradicts the fact that $\{\cdot,\cdot\}_{b}$ is positive or negative definite. Thus, non of the vector fields $\nabla_{b}^{L}f$ and $\nabla_{b}^{R}f$ admit non--trivial periodic orbits.
\end{proof}

\begin{corollary}
When $b$ is a Riemannian metric on $M$, Theorem \ref{theor1} reduces to the classical result which states that for any $F\in\mathcal{C}^{\infty}(M,\mathbb{R})$ the associated gradient vector field, $\nabla_{b}F=\nabla_{b}^{L}F=\nabla_{b}^{R}F$, does not admits non--trivial periodic orbits.
\end{corollary}

%\subsection*{Acknowledgment}
%This work was supported by a grant of the Romanian National Authority for Scientific Research, CNCS-UEFISCDI, project number PN-II-RU-TE-2011-3-0103. 

\bigskip
\bigskip

\noindent {\sc R.M. Tudoran}\\
West University of Timi\c soara\\
Faculty of Mathematics and Computer Science\\
Department of Mathematics\\
Blvd. Vasile P\^arvan, No. 4\\
300223 - Timi\c soara, Rom\^ania.\\
E-mail: {\sf razvan.tudoran@e-uvt.ro}\\
\medskip

\end{document}